\documentclass[12pt]{article}
\usepackage{amsmath}
\usepackage{latexsym}
\usepackage{amssymb}
\newtheorem{thm}{Theorem}[section]
\newtheorem{la}[thm]{Lemma}
\newtheorem{Defn}[thm]{Definition}
\newtheorem{Remark}[thm]{Remark}
\newtheorem{Conj}[thm]{Conjecture}
\newtheorem{prop}[thm]{Proposition}
\newtheorem{cor}[thm]{Corollary}
\newtheorem{Example}[thm]{Example}
\newtheorem{Number}[thm]{\!\!}
\newenvironment{defn}{\begin{Defn}\rm}{\end{Defn}}
\newenvironment{example}{\begin{Example}\rm}{\end{Example}}
\newenvironment{rem}{\begin{Remark}\rm}{\end{Remark}}

\newenvironment{numba}{\begin{Number}\rm}{\end{Number}}
\newenvironment{proof}{{\noindent\bf Proof.}}%
                  {\nopagebreak\hspace*{\fill}$\Box$\medskip\par}
\newcommand{\Punkt}{\nopagebreak\hspace*{\fill}$\Box$}
\newcommand{\wb}{\overline}
\newcommand{\ve}{\varepsilon}
\newcommand{\at}{\symbol{'100}}

\newcommand{\impl}{\Rightarrow}
\newcommand{\mto}{\mapsto}
\newcommand{\N}{{\mathbb N}}
\newcommand{\bF}{{\mathbb F}}
\newcommand{\bP}{{\mathbb P}}
\newcommand{\K}{{\mathbb K}}
\newcommand{\Q}{{\mathbb Q}}
\newcommand{\Z}{{\mathbb Z}}
\newcommand{\R}{{\mathbb R}}
\newcommand{\C}{{\mathbb C}}
\newcommand{\cg}{{\mathfrak g}}
\newcommand{\cA}{{\mathcal A}}
\newcommand{\cB}{{\mathcal B}}
\newcommand{\cC}{{\mathcal C}}
\newcommand{\cF}{{\mathcal F}}
\newcommand{\ch}{{\mathfrak h}}

\DeclareMathOperator{\Aut}{Aut}
\newcommand{\sub}{\subseteq}
\DeclareMathOperator{\id}{id}
\newcommand{\wh}{\widehat}
\DeclareMathOperator{\tor}{tor}
\DeclareMathOperator{\bil}{bil}
\DeclareMathOperator{\eq}{eq}
\DeclareMathOperator{\pr}{pr}
\DeclareMathOperator{\td}{td}
\DeclareMathOperator{\im}{im}   
\DeclareMathOperator{\op}{op}
\DeclareMathOperator{\Supp}{supp}
\begin{document}
\begin{center}
{\Large\bf Decompositions of locally compact\\[1.3mm]
contraction groups, series and extensions}\\[7mm]
{\bf Helge Gl\"{o}ckner and George A. Willis}\vspace{2mm}
\end{center}
\begin{abstract}
\hspace*{-6.3mm}A locally compact contraction group
is a pair $(G,\alpha)$, where $G$ is a locally compact group and
$\alpha\colon G\to G$ an automorphism such that $\alpha^n(x)\to e$ pointwise
as $n\to\infty$. We show that every
surjective, continuous, equivariant homomorphism
between locally compact contraction groups
admits an equivariant continuous global section.
As a consequence,
extensions of
locally compact contraction groups with abelian kernel can be
described
by continuous equivariant cohomology.
For each prime number $p$, we use 2-cocycles to construct uncountably many pairwise
non-isomorphic totally disconnected, locally compact contraction groups
$(G,\alpha)$ which are central extensions
\[
\{0\}\to \bF_p(\!(t)\!)\to G\to \bF_p(\!(t)\!)\to\{0\}
\]
of the additive group of the field of formal Laurent series over $\bF_p=\Z/p\Z$ by itself.
By contrast, there are only countably many locally compact contraction groups
(up to isomorphism) which are torsion groups and \emph{abelian},
as follows from a classification of the abelian\linebreak
locally compact contraction groups.
\end{abstract}
{\bf Classification:} Primary 22D05; 
secondary
20E22, 
20E36, 
20F18, 
20J06\\[3mm] 
{\bf\Large Introduction and statement of results}\\[4mm]
An automorphism $\alpha\colon G\to G$ of a locally compact group~$G$ is called
\emph{contractive} if $\lim_{n\to\infty}\alpha^n(x)= e$ for each $x\in G$.
The article is devoted to new aspects of the theory of such locally compact
contraction groups $(G,\alpha)$. Let $G_e$ be the connected component
of the identity element~$e$ of~$G$.
Siebert~\cite[Proposition~4.2]{Sie} showed that $G$ has an $\alpha$-stable
closed normal subgroup $G_{\td}$ such that
\[
G=G_e\times G_{\td}
\]
internally as a topological group;
moreover, $G_e$ is a simply connected, nilpotent
real Lie group (see \cite{Sie}) and hence torsion-free.
Results by the authors imply that
the set $\tor(G)$ of torsion elements
is a
closed subgroup of $G_{\td}$ and
\begin{equation}\label{fulldec}
G=G_e\times G_{p_1}\times\cdots\times G_{p_n}\times\tor(G)
\end{equation}
internally for
$\alpha$-stable closed subgroups $G_p$ of $G$ which are $p$-adic Lie groups
for some prime~$p$ (cf.\ \cite[Theorem~B]{GW}).
We improve these results by adding uniqueness statements, invariance properties
and refined decompositions:\footnote{In the following theorems, we use terminology as explained in Section~\ref{secprel}.}\\[4mm]
{\bf Theorem~A.}
\emph{For each locally compact contraction group $(G,\alpha)$, we have}:
\begin{itemize}
\item[(a)]
\emph{$G$ has a unique closed normal subgroup $G_{\td}$ such that
$G=G_e\times G_{\td}$ internally as a topological group.
The subgroup
$G_{\td}$ is
totally
dis\-connected, $\alpha$-stable,
and topologically fully invariant in~$G$.}
\item[(b)] 
\emph{The set $\tor(G)$ of torsion elements of~$G$
is a closed subgroup of~$G$ and totally disconnected.
There are a unique
$n\in\N_0$, unique prime numbers $p_1\!<\!\cdots\!<\!p_n$
and unique $p$-adic Lie groups $G_p\!\not=\!\{e\}$ for $p\!\in\!\{p_1,\ldots,p_n\}$
which are
closed
normal subgroups of~$G$ such that}
\begin{equation}\label{decoGpre}
G=G_e\times G_{p_1}\times\cdots\times G_{p_n}\times\tor(G)
\end{equation}
\emph{internally as a topological group. Each $G_p$ is topologically fully invariant
in~$G$} (\emph{and hence $\alpha$-stable}).
\item[(c)]
\emph{If $\tor(G)$ is locally pro-nilpotent, then, for each prime number~$p$,
the set $\tor_p(G)$ of $p$-torsion
elements of~$G$
is a fully invariant closed subgroup of~$G$ which is
locally
pro-$p$.
Moreover, $\tor_p(G)\not=\{e\}$ for
only finitely many~$p$, say for $p$ among the prime numbers
$q_1<q_2<\cdots < q_m$,
and}
\[
\tor(G)=\tor_{q_1}(G)\times\cdots\times \tor_{q_m}(G)
\]
\emph{internally as a topological group.}\vspace{2mm}
\end{itemize}
The structure of $G_e$ and the $G_p$ is well understood:
Siebert showed that the Lie algebra of the simply connected nilpotent
real Lie group $G_e$ admits a positive graduation
(and any real Lie algebra with this property arises in this way), see \cite{Sie}.
J.\,S.\,P. Wang showed that $G_p$ is the group of $\Q_p$-rational points of a unipotent algebraic group
defined over the field $\Q_p$ of $p$-adic numbers (and hence nilpotent) \cite{Wan}; its Lie algebra admits an $\N$-graduation
(and any $p$-adic Lie algebra with this property arises in this way; see also~\cite{Glo}).
By contrast, the structure of the torsion part $\tor(G)$ is as yet not well understood,
except that, as a locally compact contraction group,  $\tor(G)$ has a composition series
whose simple quotients are isomorphic to restricted products
\[
F^{(-\N)}\times F^{\N_0}
\]
with the right shift for $F$ a finite simple group, where $F^{\N_0}$
is equipped with the compact product topology and where the group $F^{(-\N)}$
of finitely supported sequences $(x_n)_{n\in-\N}$ in $F$ (viz.\ $x_n=e$
for all sufficiently small $n$) is given the discrete topology (see \cite{GW}).
Note that $F^{(-\N)}\times F^{\N_0}$ is isomorphic to the additive group
of the field $\bF_p(\!(t)\!)$ 
of formal Laurent series with coefficients in~$\bF_p$
(together with the contractive automorphism $\alpha\colon f(t)\mto tf(t)$)
if $F:=\bF_p$ is a cyclic group of prime order~$p$.
As there are (up to isomorphism) only countably many
possible composition factors, it is natural to ask
whether, up to isomorphism, there are only countably many torsion
contraction groups. We show that the answer is negative (see Theorem~\ref{mainres}):\\[4mm]
{\bf Theorem B.}
\emph{For each prime $p$, there is an uncountable set of contraction groups $(H,\beta)$ which are pairwise
non-isomorphic as contraction groups and such that each $(H,\beta)$ is a central extension}
\[
\{0\}\to \bF_p(\!(t)\!)\to H\to \bF_p(\!(t)\!)\to\{0\}.\vspace{2mm}
\]
To this end, we describe continuous $2$-cocycles $\eta_s\colon \bF_p(\!(t)\!)\times\bF_p(\!(t)\!)\to
\bF_p(\!(t)\!)$ with $\alpha\circ \eta_s=\eta_s \circ (\alpha\times\alpha)$,
parametrized by sequences $s\in\{0,1\}^\N$, which give rise to contraction groups
\[
\bF_p(\!(t)\!)\times_{\eta_s}\bF_p(\!(t)\!)
\]
which are pairwise non-isomorphic for different values of~$s$.
This procedure is quite natural, since all central extensions of locally
compact contraction groups can be
described by equivariant, continuous $2$-cocycles (as observed in Appendix~\ref{cohom}),
as a consequence
of the following result:\pagebreak

\noindent
{\bf Theorem C.} \emph{If $(G,\alpha)$ and $(H,\beta)$ are locally compact
contraction groups and $q\colon G\to H$ is a surjective, continuous homomorphism such that
$\beta\circ q=q\circ\alpha$, then there exists a continuous map
$\sigma\colon H\to G$ such that $\sigma(e)=e$, $q\circ \sigma=\id_H$ and $\alpha\circ\sigma=\sigma\circ\beta$.}\\[4mm]
In other words, all extensions of locally compact
contraction groups admit
equivariant, continuous global sections.\\[2.3mm]
Theorem~B provides uncountably many non-isomorphic contraction
groups which are central extensions of the abelian group $\bF_p(\!(t)\!)$
by itself, and hence nilpotent. By contrast, there are only two non-isomorphic
\emph{abelian} contraction groups which are extensions of $\bF_p(\!(t)\!)$
by itself. More generally, we have the following classification
(which subsumes the preceding assertion with $F=\Z/4\Z$ and the Klein 4-group
$F=(\Z/2\Z)^2$):\\[4mm]
{\bf Theorem D.} \emph{Let $(G,\alpha)$ be a totally disconnected,
locally compact contraction group such that $G$ is a torsion group and abelian.
Then $(G,\alpha)$ is isomorphic to $F^{(-\N)}\times F^{\N_0}$ with the right shift
for some finite abelian group~$F$, which is determined up to isomorphism.
Conversely, every such contraction group
is an abelian torsion group.}\\[4mm]
Notably, up to isomorphism there are only countably many locally compact abelian
contraction groups which are torsion groups (see also Corollary~\ref{thetorso}).\\[2.3mm]
Based on Theorem~D, we obtain a classification of the locally compact abelian
contraction groups. To formulate it, let $\bP$ be the set of prime numbers
and $\Q_\infty:=\R$. For $p\in\bP\cup\{\infty\}$,
let $\wb{\Q_p}$ be an algebraic closure of~$\Q_p$
and $\Omega_p$ be the set of all monic irreducible polynomials
$f$ over~$\Q_p$ such that $|\lambda|<1$ for all zeros $\lambda$ of~$f$
in~$\wb{\Q_p}$, where $|.|$ is the unique extension of the usual absolute
value on~$\Q_p$ to an absolute value on~$\wb{\Q_p}$.
Given a monic irreducible polynomial $f\in\Q_p[X]$ and $n\in\N$, we endow the $\Q_p$-vector space
\[
E_{f^n}:=\Q_p[X]/f^n\Q_p[X]
\]
with the $\Q_p$-linear automorphism~$\alpha_{f^n}$
taking $g+f^n\Q_p[X]$ to $Xg+f^n\Q_p[X]$ for all polynomials $g\in\Q_p[X]$
(whose characteristic and minimal polynomials
both are~$f^n$).
Given $p\in \bP$ and $n\in\N$, let $F_{p^n}:=\Z/p^n\Z$
and write $F_{p^n}(\!(t)\!)$
as a shorthand for $F_{p^n}^{(-\N)}\times F_{p^n}^{\N_0}$, endowed
with the right shift $\alpha_{p^n}$ (see also \ref{morecompact}).\\[3mm]
We obtain the following Structure Theorem for Locally Compact
Abelian Contraction Groups.\pagebreak

\noindent
{\bf Theorem E.}
\emph{Let $(G,+)$ be a locally compact abelian group and $\alpha\colon G\to G$
be a contractive automorphism. Then $(G,\alpha)$ is isomorphic to}
\[
\qquad \qquad \bigoplus_{p\in\bP\cup\{\infty\}} \bigoplus_{f\in \Omega_p}\;\bigoplus_{n\in\N}\,
(E_{f^n},\alpha_{f^n})^{\mu(p,f,n)}\, \oplus\;\,
\bigoplus_{p\in\bP} \bigoplus_{n\in\N} \, (F_{p^n}(\!(t)\!),\alpha_{p^n})^{\nu(p,n)}\quad
(*)
\]
\emph{as a contraction group, for uniquely determined $\mu(p,f,n)\in \N_0$
which are non-zero for only finitely many
$(p,f,n)$ with $(p,n)\in (\bP\cup\{\infty\})\times\N$ and $f\in\Omega_p$,
and uniquely determined $\nu(p,n)\in\N_0$
which are non-zero for only finitely many $(p,n)\in\bP\times\N$.
Conversely, all groups of the form $(*)$ are locally compact abelian contraction
groups.}\\[4mm]
We are also able to adapt various results to
general locally compact contraction groups $(G,\alpha)$ which were
previously only known for totally disconnected~$G$.
For example, generalizing \cite[Theorem~3.3]{GW}
we show that every locally compact contraction group admits a composition
series
\[
\{e\}=G_0\lhd G_1\lhd\cdots\lhd G_n=G
\]
as a contraction group, and that a Jordan-H\"{o}lder Theorem holds for such
(see Theorem~\ref{JH}).
As the composition factors $G_j/G_{j-1}$ are simple
contraction groups (in the sense of~\ref{simpdef}), it is of interest to classify the latter.
A classification of the simple totally disconnected contraction groups
was obtained in \cite[Theorem~A]{GW}. The following result completes the picture:\\[4mm]
{\bf Theorem~F.}
\emph{If $(G,\alpha)$ is a simple locally compact contraction group,
then~$G$ is either connected or totally disconnected.
The family $(E_f,\alpha_f)_{f\in\Omega_\infty}$
is a\linebreak
system of representatives for the isomorphism classes of the connected\linebreak
simple locally compact contraction groups.}\\[4mm]
It is known from \cite[Corollary~3.6]{GW} that a continuous group
homomorphism $\phi\colon G\to H$ has closed image and is an open mapping
onto its image\linebreak
whenever $(G,\alpha)$ and $(H,\beta)$
are totally disconnected, locally compact contraction groups
and $\beta\circ\phi=\phi\circ\alpha$.
We obtain the following generalization (see Proposition~\ref{nicehomtf}
and Corollary~\ref{precloprod}):\\[4mm]
{\bf Theorem~G.}
\emph{Let $(G,\alpha)$ and $(H,\beta)$ be locally compact
contraction groups and $\phi\colon G\to H$ be a continuous group homomorphism.
Then we have}:
\begin{itemize}
\item[(a)]
\emph{$\phi(G)$ is closed in~$H$ if and only if $\phi(\tor(G))$
is closed in~$H$, in which case the co-restriction $\phi|^{\phi(G)}\colon G\to\phi(G)$
is an open map.}
\item[(b)]
\emph{If $\beta\circ\phi=\phi\circ\alpha$, then $\phi(G)$ is closed in~$H$.}\vspace{1mm}
\end{itemize}
{\bf Structure of the article.} In Section~\ref{secprel},
we fix notation and describe basic facts and conventions for later use.
Section~\ref{secdecomp} is devoted to uniqueness
properties and
Theorem~A.
Section~\ref{secsec} contains the proof of Theorem~C.
As\linebreak
already mentioned,
the latter theorem allows central extensions of locally\linebreak
compact
contraction groups to be described in terms
of suitably-defined equivariant continuous $2$-cocycles (and likewise for extensions
with abelian kernel). The arguments,
which closely parallel the classical case of group extensions,
are sketched in Appendix~\ref{cohom}.
Theorems~F and~G are established in Section~\ref{newsec}.
In Section~\ref{secab}, we prove Theorems~D and~E.
Section~\ref{secauxi} provides two auxiliary results concerning
extensions of contraction groups.
In Section~\ref{secbiadd}, we turn to a special class of 2-cocycles
on contraction groups leading to central extensions,
namely equivariant, \emph{biadditive}, continuous mappings $\omega\colon \bF_p(\!(t)\!)\times
\bF_p(\!(t)\!)\to\bF_p(\!(t)\!)$.
We obtain an explicit parametrization of these in terms of two-sided
sequences of elements in~$\bF_p(\!(t)\!)$.
In Section~\ref{secex}, we show that a suitable uncountable set of parameters
gives rise to extensions which are not only inequivalent extensions,
but pairwise non-isomorphic as contraction groups (thus proving Theorem~B).
The article closes with open problems concerning torsion contraction groups
which are locally pro-$p$ (Section~\ref{secnilp}).\\[2.3mm]
{\bf Acknowledgements.}
The first author is grateful to George A. Willis and the University of Newcastle, N.S.W.
for support which enabled research\linebreak
visits in October 2015 and September~2017.
Partial support also came from the Deutsche Forschungsgemeinschaft, DFG grant GL 357/10-1.
\section{Preliminaries and notation}\label{secprel}
We write $\Z$ for the
ring of
integers, $\N:=\{1,2,\ldots\}$ and $\N_0:=\N\cup\{0\}$.
If $f \colon X\to X$ is a self-map of a set~$X$,
then a subset $M\sub X$ is called \emph{$\alpha$-stable}
if $f(M)=M$. If $M$ is a subset of a group~$G$
and $\alpha(M)\sub M$ for each endomorphism~$\alpha$ of~$G$,
then~$M$ is called \emph{fully invariant}.
If $G$ is a topological group and $\alpha(M)\sub M$
for each continuous endomorphism~$\alpha$ of~$G$,
then~$M$ is called \emph{topologically fully invariant}.
As usual, we write $\langle M\rangle$ for the subgroup
generated by a subset~$M$ of a group~$G$.
If $A$ and $B$ are subgroups of a group~$G$,
we write $[A,B]$ for the subgroup of~$G$ generated by
$\{aba^{-1}b^{-1}\colon a\in A,b\in B\}$,
as usual.
We define $C^1(G):=G$ and $C^{n+1}(G):=[G,C^n(G)]$ for $n\in\N$.
We define $G^{(0)}:=G$ and $G^{(n+1)}:=[G^{(n)},G^{(n)}]$
for $n\in\N_0$.
\begin{numba}
If $G$ is a group and $p$ a prime number, we write
\[
\tor_p(G):=\{g\in G\colon (\exists n\in\N)\; g^{p^n}=e\}
\]
for the subset of $p$-torsion elements. If $\tor_p(G)$ is a subgroup,
we call it the \emph{$p$-torsion subgroup} of~$G$.
\end{numba}
\begin{numba}\label{fuinva}
If $G$ and $H$ are groups and $\phi\colon G\to H$ is a group
homomorphism, then
\[
\phi(\tor_p(G))\sub\tor_p(H)
\]
for each prime number~$p$.
In particular, $\tor_p(G)$ is fully invariant in~$G$.
\end{numba}
\begin{numba}\label{morecompact}
If $p$ is a prime number, we let $\bF_p$ be the finite field of order~$p$
and $\bF_p(\!(t)\!)$ be the valued field of formal Laurent series with coefficients in~$\bF_p$
(cf.\ \cite{Sch},~\cite{Wei});
its elements are series of the form $x=\sum_{n=N}^\infty x_nt^n$
with $N\in\Z$ and $x_n\in \bF_p$ for integers $n\geq N$.
We shall use the absolute value $|.|$ on $\bF_p(\!(t)\!)$ given by
\[
|x|:=p^{-N}
\]
if $x\not=0$ and $N$ (as above) is chosen minimal with $x_N\not=0$.
Then
\begin{equation}\label{ultra}
|x+y|=|x|\quad\mbox{for all $x,y\in\bF_p(\!(t)\!)$ such that $|y|<|x|$.}
\end{equation}
In the same way, we can consider the ring $F(\!(t)\!)$
of formal Laurent series over a finite commutative ring~$F$.
Its additive group with the contractive automorphism $x\mto tx$
is isomorphic to $F^{(-\N)}\times F^{\N_0}$ with the right shift,
and we appreciate the more compact notation.
\end{numba}
All topological groups we consider are assumed Hausdorff.
As usual, we use \emph{locally compact group}
as a shorthand for locally compact topological group.
\begin{numba}
If we speak of an \emph{automorphism}~$\alpha$
of a locally compact group~$A$, then $\alpha$ and its inverse are assumed
continuous. We write $\Aut(A)$ for the group of all such automorphisms,
and endow it with the topology induced from the Braconnier
topology on the set $\text{Homeo}(A)$ of auto-homeomorphisms
of~$A$\linebreak
(called the `refined compact-open topology' in~\cite{Str}).
A subbasis for this\linebreak
topology is given by the sets
\[
\{\alpha\in\Aut(A)\colon \alpha(K_1)\sub U_1\;\mbox{and}\; \alpha^{-1}(K_2) \sub U_2\},
\]
for $K_1, K_2$ in the set of compact subsets of~$A$ and open subsets $U_1,U_2\sub A$.
\end{numba}
\begin{numba}\label{goodStr}
If~$G$ is a topological group, then a homomorphism $\gamma \colon G\to\Aut(A)$
is continuous if and only if the corresponding left $G$-action
\[
G\times A\to A,\quad (g,a)\mto\gamma(g)(a)
\]
is continuous (see, e.g., \cite[Lemma~10.4]{Str}).
\end{numba}
\begin{numba}\label{simpdef}
An automorphism $\alpha$ of a topological group~$G$
is called \emph{contractive} if
\[
\lim_{n\to\infty}\alpha^n(g)=e\quad\mbox{for all $g\in G$.}\vspace{-1mm}
\]
A \emph{locally compact contraction group}
is a pair $(G,\alpha)$ consisting of a
locally compact group~$G$ and a contractive automorphism $\alpha\colon G\to G$.
If $G$ is, moreover, totally disconnected (resp., connected),
then $(G,\alpha)$ is called a totally disconnected (resp., connected)
locally compact contraction group.
We call $(G,\alpha)$ a \emph{simple} contraction group
if $G\not=\{e\}$ and $G$ does not have $\alpha$-stable, closed normal
subgroups except for $\{e\}$ and~$G$.
\end{numba}
\begin{numba}\label{compacontra}
Every contractive automorphism $\alpha$
of a locally compact group~$G$
is \emph{compactly contractive} in the following sense:
For each compact set $K\sub G$ and identity neighbourhood $U\sub G$,
there exists $N\in\N$ such that
\[
\alpha^n(K)\sub U\quad\mbox{for all integers $n\geq N$}
\]
(see
\cite[Proposition 2.1]{Wan} or \cite[Lemma~1.4\,(iv)]{Sie}).
Thus $\{\alpha^n(V)\colon n\in\N_0\}$ is a basis of identity neighbourhoods
in~$G$, for each compact identity neighbourhood $V\sub G$.
\end{numba}
\begin{numba}\label{inv-sub}
If $(G,\alpha)$ is a totally disconnected, locally compact contraction group,
then there exists a compact open subgroup
$V\sub G$ such that $\alpha(V)\sub V$ (see
\cite[Lemma~3.2\,(i)]{Sie}).
\end{numba}
\begin{numba}
If $G$ is a locally compact group with Haar measure~$\lambda_G$
and $\alpha\colon G\to G$\linebreak
an automorphism, there is
$\Delta(\alpha)\in\;]0,\infty[$ (the \emph{module} of~$\alpha$) such that
$\lambda_G(\alpha(A))=\Delta(\alpha)\lambda_G(A)$
for all Borel sets $A\sub G$ (cf.\ \cite{HaR}).
If $G$ is totally disconnected and $\alpha$ is contractive,
let $V\sub G$ be a compact open subgroup such that $\alpha(V)\sub V$ (see \ref{inv-sub}).
Then
\begin{equation}\label{formul-mod}
\Delta(\alpha^{-1})=[V:\alpha(V)].
\end{equation}
\end{numba}
\begin{numba}
Let $(G,\alpha_G)$ and $(H,\alpha_H)$ be locally compact
contraction groups. A continuous group homomorphism $\phi \colon G\to H$
is called a \emph{morphism of contraction groups}
if it is \emph{equivariant} in the sense that
\[
\alpha_H\circ\phi=\phi\circ\alpha_G;
\]
we then write $\phi\colon (G,\alpha_G)\to (H,\alpha_H)$.
If, moreover, $\phi$ is invertible and also $\phi^{-1}$
is a morphism of contraction groups,
then $\phi$ is an \emph{isomorphism of contraction groups}
(and an \emph{automorphism} of the contraction group $(G,\alpha_G)$
if, moreover, $(G,\alpha_G)=(H,\alpha_H)$).
We write $\Aut(G,\alpha_G)$ for the group of automorphisms of $(G,\alpha_G)$.
\end{numba}
\begin{numba}\label{nicemorph}
If
$\phi\colon (G,\alpha_G)\to (H,\alpha_H)$ is a morphism between totally disconnected,
locally compact contraction groups,
then $\phi(G)$ is a closed, $\alpha_H$-stable subgroup of~$H$
and the co-restriction $\phi|^{\phi(G)}\colon G\to\phi(G)$
is a quotient homomorphism (i.e., surjective, continuous, and open),
see \cite[Corollary~3.6]{GW}.
\end{numba}
\begin{numba}
If $(A,\alpha_A)$, $(\widehat{G},\widehat{\alpha})$, and $(G,\alpha)$
are totally disconnected, locally compact contraction groups,
we call a short exact sequence
\[
\{e\}\to A\stackrel{\iota}{\to} \widehat{G}\stackrel{q}{\to} G\to\{e\}
\]
(or simply $\widehat{G}$) an \emph{extension of contraction groups}
if $\iota\colon (A,\alpha_A)\to (\widehat{G},\widehat{\alpha})$
and $q\colon (\widehat{G},\widehat{\alpha})\to (G,\alpha)$
are morphisms of contraction groups.
Then $\iota$ is a homeomorphism onto its image
and~$q$ is a quotient homomorphism,
by the fact recalled in~\ref{nicemorph}.
If~$A$ is abelian, then $\wh{G}$ is called an extension of~$G$ \emph{with abelian
kernel};
if $\iota(A)$ is contained in the centre $Z(\widehat{G})$ of the group
$\widehat{G}$, then $\widehat{G}$ is called a \emph{central extension} of~$G$.
\end{numba}
\begin{numba}
If $(G,\alpha)$ is a locally compact contraction group, we call a series
\begin{equation}\label{serial}
\{e\}\lhd G_0\lhd G_1\lhd\cdots\lhd G_n=G
\end{equation}
a \emph{series of contraction groups}\footnote{In the terminology of~\cite[Definition~2.1]{GW},
this is an $\langle\alpha\rangle$-series.}
if $G_j$ is a closed $\alpha$-stable subgroup of~$G$
for each $j\in\{0,\ldots, n\}$. We endow $G_j$ with the contractive automorphism
$\alpha|_{G_j}$ for $j\in\{0,\ldots, n\}$. Moreover, we endow
$G_j/G_{j-1}$ with the contractive automorphism $\wb{\alpha}_j\colon xG_{j-1}\mto\alpha(x)G_{j-1}$
for $j\in\{1,\ldots, n\}$. If $(G_j/G_{j-1},\wb{\alpha}_j)$ is a simple
contraction group for all $j\in\{1,\ldots, n\}$, then (\ref{serial})
is called a \emph{composition series} of contraction groups.
\end{numba} 
\begin{numba}
Following \cite[Definition~1.7]{Sie},
we say that a locally compact group~$G$ is \emph{contractible}
if it admits a contractive automorphism $\alpha\colon G\to G$.
\end{numba}
Repeatedly, we shall construct continuous functions as follows.
\begin{la}\label{egsum}
Let $X$ be a topological space, $G$ be a totally disconnected,
locally compact group
which is metrizable and $(f_n)_{n\in\N_0}$ be a sequence of continuous functions
$f_n\colon X\to G$. Assume that there exists a basis of identity neighbourhoods
$U_1\supseteq U_2\supseteq\cdots$ in~$G$ such that $U_n$ is a compact open
subgroup of~$G$ and $f_n(X)\sub U_n$, for each $n\in\N$.
Then
\[
g_n:=f_0f_1\ldots f_n\colon X\to G,\quad x\mto f_0(x)f_1(x)\cdots f_n(x)
\]
converges uniformly to a continuous function $g\colon X\to G$.
\end{la}
\begin{proof}
For all $n\in\N_0$ and $m\in \N_0$, we have
\begin{equation}\label{telesc}
(g_{n+m}(x))^{-1}g_n(x)=(f_{n+m}(x))^{-1}\cdots (f_{n+1}(x))^{-1}\in U_{n+1}\;
\mbox{for all $x\in X$.}
\end{equation}
For fixed $x\in X$, we deduce that $(g_n(x))_{n\in\N}$ is a Cauchy sequence in~$G$
and hence convergent (as locally compact groups are complete).
Let
$g(x)$ be the limit.
Letting $m\to\infty$ in (\ref{telesc}), we see that
\[
(g(x))^{-1}g_n(x)\in U_{n+1}\;
\mbox{for all $n\in\N_0$ and $x\in X$.}
\]
Thus $g_n\to g$ uniformly, whence $g$ is continuous.
\end{proof}
\begin{numba}
Let $\K=\R$ or $\K=\Q_p$, respectively, for some prime~$p$.
Recall that a real (resp.\ $p$-adic) Lie group
is a group~$G$, endowed with a finite-dimensional $\K$-analytic 
manifold structure which makes the group operations analytic.
Then the tangent space $L(G):=T_e(G)$ of~$G$
at the neutral element is a Lie algebra over~$\K$
in a natural way. If $\phi\colon G\to H$ is a continuous
group homomorphism, then $\phi$ is analytic
and we write $L(\phi):=T_e(\phi)$ for its tangent map at~$e$,
which is a Lie algebra homomorphism from $L(G)$ to $L(H)$.
In particular, there is at most one real (resp.,
$p$-adic) Lie group structure on a given locally compact group~$G$.
As usual, we identify a real (resp.\ $p$-adic) Lie group
with its underlying locally compact group
(see~\cite{Bou} and~\cite{Ser} for further information).
\end{numba}
For generalities concerning pro-finite groups and pro-$p$-groups,
the reader is referred to~\cite{Wil} and~\cite{Dix}.
\begin{defn}
Let $p$ be a prime.
We say that a locally compact group~$G$ is \emph{locally pro-$p$}
if it has a compact open subgroup $U$ which is a pro-$p$-group.
\end{defn}
For example, every $p$-adic Lie group is locally pro-$p$
(cf.\ \cite[Theorems~8.29 and 8.31]{Dix}).
\begin{defn}
We say that a locally compact group~$G$ is \emph{locally pro-nilpotent}
if it has a compact open subgroup~$U$ which is pro-nilpotent
(i.e., a projective limit of finite nilpotent groups).
\end{defn}
For example, every totally disconnected, \emph{abelian} locally compact contraction group
$(G,\alpha)$ is locally pro-nilpotent, as every compact open subgroup
$U\sub G$ is abelian and hence pro-nilpotent.
\section{Decompositions of contraction groups}\label{secdecomp}
In this section, we prove Theorem~A
and provide additional results concerning decompositions
of locally compact contraction groups.\\[2.3mm]
We first pinpoint properties of group elements in special classes of contraction
groups (like totally disconnected groups, connected groups, $p$-adic Lie groups, and torsion groups).
These properties will help us to distinguish elements in the individual factors
of decompositions of contraction groups, and hence to
see that the factors are unique.\\[5mm]
{\bf Divisible elements; topologically periodic elements}
\begin{numba}\label{diviper}
Recall that an element $g$ in a group~$G$
is called \emph{divisible} if, for each
$n\in\N$, there exists $x\in G$ such that $x^n=g$.
If~$x$ is uniquely determined for each~$n$,
then $g$ is called \emph{uniquely divisible}.
An element $g$ of a locally compact group~$G$
is called
\emph{topologically periodic}
if $g$ generates a relatively compact subgroup.
If $\phi\colon G\to H$
is a continuous homomorphism between locally compact groups
and $g\in G$ is topologically periodic,
then $\phi(g)$ is topologically periodic in~$H$.
\end{numba}
\begin{la}\label{tdtopper}
If $(G,\alpha)$ is a totally disconnected,
locally compact contraction group, then $G$ is a union
of compact open subgroups. In particular,
each $g\in G$
is topologically periodic.
\end{la}
\begin{proof}
Let $V$ be a compact open subgroup of~$G$.
Then $G=\bigcup_{n\in\N_0}\alpha^{-n}(V)$,
from which the assertions follow.\vspace{4mm}
\end{proof}
{\bf Basic properties of connected contraction groups}\\[2.3mm]
Results by Siebert~\cite{Sie} imply:
\begin{la}\label{conncase}
If $(G,\alpha)$ is a connected locally compact contraction group,
then $G$ is a simply connected, nilpotent real Lie group,
$L(G)$ is a nilpotent real Lie algebra,
$L(\alpha)\colon L(G)\to L(G)$ is contractive and we have:
\begin{itemize}
\item[\rm(a)]
$(G,\alpha)$ is isomorphic to $((L(G),*),L(\alpha))$
as a contraction group, where $*\colon L(G)\times L(G)\to L(G)$,
$(x,y)\mto x*y=x+y+\frac{1}{2}[x,y]+\cdots$
is the Baker-Campbell-Hausdorff {\rm(BCH)-} multiplication.
\item[\rm(b)]
$G$ is torsion-free and each $g\in G$ is uniquely divisible.
\item[\rm(c)]
If $g\in G$ is topologically periodic, then $g=e$.
\item[\rm(d)]
For each $e\not=g\in G$ and each sequence $(k_n)_{n\in\N}$ in $\N$ such that $k_n\to\infty$,
the sequence $(g^{k_n})_{n\in\N}$ does not converge in~$G$.
\end{itemize}
\end{la}
\begin{proof}
By \cite[2.3 and Corollary~2.4]{Sie},
$G$ is a simply connected real Lie group whose Lie algebra $L(G)$ is nilpotent.
Hence~$G$ is nilpotent (cf.\ Proposition~12 in \cite[Chapter~III, \S9.5]{Bou}).
Since $L(G)$ is nilpotent,
the BCH-series
converges on all of $L(G)\times L(G)$
and makes $L(G)$ a Lie group with neutral element~$0$
(see explanations after Proposition~12 in \cite[Chapter~III, \S9.5]{Bou}).
Since $G$ is simply connected and nilpotent,
the exponential map $\exp_G\colon L(G)\to G$
is an isomorphism of Lie groups from $(L(G),*)$ to~$G$,
by Proposition~13 in \cite[Chapter~III, \S9.5]{Bou}.
Since
\begin{equation}\label{naturali}
\exp_G\circ L(\alpha)=\alpha\circ \exp_G
\end{equation}
by Proposition~10 in \cite[Chapter~III, \S6.4]{Bou},
we see that the automorphism $L(\alpha)\colon L(G)\to L(G)$
of $(L(G),*)$ (and of $(L(G),+)$) is contractive. By~(\ref{naturali}), $\exp_G$ is an isomorphism
of contraction groups from $((L(G),*),L(\alpha))$ to $(G,\alpha)$,
establishing~(a).\\[2.3mm]
Using the BCH-formula, for $x\in L(G)$ we find that the map $\R\to L(G)$, $t\mto tx$
is a continuous group homomorphism to $(L(G),*)$, whence $\R x$ is a subgroup of $(L(G),*)$
and
\begin{equation}\label{powermult}
x^n=nx\quad\mbox{for all $\,n\in\Z$.}
\end{equation}
Since $(L(G),+)$ is torsion-free and each element of this group
is uniquely divisible, we deduce with~(\ref{powermult})
that also the group $(L(G),*)$
is torsion-free and each of its elements is uniquely divisible (entailing~(b)).
If $0\not=x\in L(G)$ was topologically periodic in $(L(G),*)$,
then the discrete group $\langle x\rangle=\Z x\cong\Z$ would be relatively
compact in $\R x\cong\R$, a contradiction. Thus~(c) holds.\\[2.3mm]
To establish~(d), let $\|.\|$ be a norm on~$L(G)$. If $0\not=x\in L(G)$ and $k_n\to\infty$,
then $\|x^{k_n}\|=\|k_nx\|=k_n\|x\|\to\infty$, whence $x^{k_n}$
cannot converge in $L(G)$.\vspace{4mm}
\end{proof}
{\bf Basic properties of {\boldmath$p$-adic} contraction groups}\\[2.3mm]
For $p$-adic contraction groups, a counterpart of Lemma~\ref{conncase}
is available.
\begin{la}\label{padcase}
If $(G,\alpha)$ is a locally compact contraction group such that~$G$
is a $p$-adic Lie group, then $G$ is nilpotent,
$L(G)$ is a nilpotent $p$-adic Lie algebra,
$L(\alpha)\colon L(G)\to L(G)$ is contractive and we have:
\begin{itemize}
\item[\rm(a)]
$(G,\alpha)$ is isomorphic to $((L(G),*),L(\alpha))$
as a contraction group, where $*\colon L(G)\times L(G)\to L(G)$
is the BCH-multiplication.
\item[\rm(b)]
$G$ is torsion-free and each $g\in G$ is uniquely divisible.
\item[\rm(c)]
For each $g\in G$, we have $\lim_{n\to\infty}g^{p^n}=e$.
\end{itemize}
\end{la}
\begin{proof}
By \cite[Theorem~3.5 (ii) and (iv)]{Wan},
$G$ is nilpotent and $L(\alpha)$ is contractive.
Hence $L(G)$ is a nilpotent $p$-adic Lie algebra (by Proposition~12 in \cite[Chapter~III, \S9.4]{Bou}),
ensuring that the BCH-series
converges on all of $L(G)\times L(G)$
and makes $L(G)$ a $p$-adic Lie group
(see explanations after Proposition~12 in \cite[Chapter~III, \S9.5]{Bou}).
The Lie algebra automorphism $L(\alpha)$
also is an automorphism of $(L(G),*)$
(as the BCH-series is given by nested Lie brackets).
There exists a compact open subgroup $V$ of $(L(G),*)$
and an exponential function $\exp_G\colon V\to G$
whose image~$U$ is a compact open subgroup of~$G$,
such that $\exp_G\colon (V,*)\to U$ is an isomorphism of $p$-adic Lie groups,
and such that
\[
\exp_G\circ L(\alpha)|_W=\alpha\circ\exp_G|_W
\]
for some compact open subgroup~$W$ of $(V,*)$ such that $L(\alpha)(W)\sub V$
(cf.\ Proposition~8 in \cite[Chapter~III, \S4.4]{Bou}).
Since~$L(\alpha)$ is contractive, there exists a compact open subgroup
$Q$ of $(L(G),*)$ such that $Q\sub W$ and $L(\alpha)(Q)\sub Q$ (cf.\ \ref{inv-sub} and
\ref{compacontra});
after replacing both $V$ and $W$ with~$Q$, we may assume that
$L(\alpha)(V)\sub V$ and $\exp_G\circ L(\alpha)|_V=\alpha\circ \exp_G$.
Now \cite[Proposition~2.2]{Wan} provides an isomorphism
$\phi\colon ((L(G),*),L(\alpha))\to (G,\alpha)$ of contraction groups
which extends~$\exp_G$ (establishing~(a)).\\[2.3mm]
Since $\Q_p\to L(G)$, $t\mto tx$ is a continuous homomorphism to $(L(G),*)$
for each $x\in L(G)$, we see that (\ref{powermult}) holds
and deduce as in the proof of Lemma~\ref{conncase}
that $(L(G),*)$ is torsion-free and all of its
elements are uniquely divisible (as both properties are inherited from $(L(G),+)$).
Thus~(b) holds.\\[2.3mm]
Using~(\ref{powermult}), we see that $x^{p^n}=p^nx\to 0$
as $n\to\infty$ for each element $x$ of $(L(G),*)$, from which~(c) follows.
\end{proof}
\begin{la}\label{torfreepad}
Let $p\not=q$ be prime numbers,
$G$ be a torsion-free $p$-adic Lie group $($e.g.,
a contractible $p$-adic Lie group$)$,
and $e\not=g\in G$. Then
the sequence $(g^{q^n})_{n\in\N}$ does not converge to~$e$ in~$G$.
\end{la}
\begin{proof}
There exists an open $0$-neighbourhood $V\sub L(G)$
such that the BCH-series converges on~$V\times V$
and makes~$V$ a $p$-adic Lie group $(V,*)$, and an open subgroup $U\sub G$
which is isomorphic to $(V,*)$ (cf.\ Theorem~2 and Lemma~3 in \cite[Chapter III, \S4.2]{Bou}).
Let $\phi\colon U\to V$ be an isomorphism.
If there were $g\in G\setminus\{e\}$ such that $g^{q^n}\to e$,
then we could find $m\in \N$ such that $g^{q^n}\in U$ for all $n\geq m$.
Since~$G$ is torsion-free, $g^{q^m}\not=e$ and thus $x:=\phi(g^{q^m})\not=0$.
Choose a norm $\|.\|$ on the $p$-adic vector
space~$L(G)$.
We have $x^{q^n}=\phi(g^{q^{n+m}})\to 0$ as $n\to\infty$.
But $q$ has $p$-adic absolute value $|q|_p=1$,
as is well-known.\footnote{We have $ap+bq=1$ for suitable integers $a$ and $b$,
whence $bq=1-ap$ and hence $|b|_p|q|_p=1$. Since $|b|_p,|q|_p\leq 1$,
the latter entails $|q|_p=1$.}
Hence $\|x^{q^n}\|=\|q^n x\|=|q|_p^n\|x\|=\|x\|$ for all $n\in\N$
and thus $x^{q^n}\not\to 0$, contradiction.\vspace{4mm}
\end{proof}
{\bf Basic properties of torsion contraction groups}\vspace{.3mm}
\begin{la}\label{torisgp}
Let $(G,\alpha)$ be a locally compact contraction group
and $\tor(G)$ be the set of all torsion elements of~$G$.
Then $\tor(G)$ is a fully invariant closed subgroup of~$G$
and totally disconnected.
\end{la}
\begin{proof}
By \cite[Proposition~4.2]{Sie},
$G$ has a totally disconnected, $\alpha$-stable closed normal subgroup~$D$
such that $G=G_e\times D$ internally as a topological group.
Since $G_e$ is torsion-free by Lemma~\ref{conncase},
we have $\tor(G)=\tor(D)$.
As $\tor(G)=\tor(D)$ is a closed subgroup of~$D$ by \cite[Theorem~B]{GW},
it also is a closed subgroup of~$G$ (and fully invariant
by a trivial argument).
\end{proof}
\begin{la}\label{torsioncase}
If $(G,\alpha)$ is a locally compact contraction group
and $G$ is a torsion group,
then $G$ is totally disconnected and has finite exponent.
If an element $g\in G$ is divisible, then $g=e$.
\end{la}
\begin{proof}
By Lemma~\ref{torisgp}, $G$ is totally disconnected.
The torsion group~$G$ has finite exponent, as
there exists a composition series $\{e\}=G_0\lhd G_1\lhd\cdots\lhd G_n=G$
of closed $\alpha$-stable subgroups
(see \cite[Theorem~3.3]{GW}) and each of the factors $G_j/G_{j-1}$
is a torsion group of finite exponent (cf.\ \cite[Theorem~A\,(a)]{GW}).
Being a torsion group of finite exponent, $G$ does not contain non-trivial
divisible elements.\vspace{5mm}
\end{proof}
{\bf Basic properties of contraction groups which are locally {\boldmath pro-$p$}}
\begin{la}\label{unionprop}
Let $(G,\alpha)$ be a totally disconnected,
locally compact contraction group.
If $G$ is locally pro-$p$ for some prime~$p$,
then every compact subgroup of~$G$ is a pro-$p$-group;
moreover,
$G$ is a union of compact open subgroups
which are pro-$p$-groups.
\end{la}
\begin{proof}
Let $K$ be a compact subgroup of~$G$ and $U$ be a compact open subgroup
of~$G$ which is a pro-$p$-group. 
Since~$\alpha$ is compactly contractive, we have $\alpha^n(K)\sub U$ for some $n\in\N$,
whence~$\alpha^n(K)$ is a pro-$p$-group by \cite[Proposition~1.11\,(i)]{Dix}.
Since $\alpha^n|_K\colon K\to \alpha^n(K)$ is an isomorphism,
also~$K$ is a pro-$p$-group.
By Lemma~\ref{tdtopper},
$G$ is a union of compact open subgroups;
any such is a pro-$p$-group, as just explained.
\end{proof}
We shall use the following well-known fact:
\begin{la}\label{homspropq}
Let $p\not=q$ be primes,
$P$ be a pro-$p$-group and $Q$ be a pro-$q$-group.
If $\phi\colon P\to Q$ is a continuous
group homomorphism, then $\phi(x)=e$ for all $x\in P$.
\end{la}
\begin{proof}
On the one hand, $\phi(P)$ is a pro-$p$-group,
being isomorphic to $P/\ker(\phi)$ (see \cite[Proposition~1.11\,(b)]{Dix}).
On the other hand, $\phi(P)$ is a pro-$q$-group, being
a compact subgroup of~$Q$ (see \cite[Proposition~1.11\,(a)]{Dix}). Hence $\phi(P)/U$
is both a $p$-group and a $q$-group
and thus $\phi(P)/U=\{e\}$,
for each open normal subgroup $U\sub \phi(P)$. As a consequence,
$\phi(P)=\{e\}$.
\end{proof}
$\;$\vspace{-8mm}\pagebreak

\noindent
{\bf Continuous group homomorphisms between contraction groups}
\begin{la}\label{opmaphere}
Let $\phi\colon G\to H$ be a surjective, continuous
group homomorphism between contractible locally compact groups~$G$
and~$H$. Then $\phi$ is an open map.
If~$\phi$ is bijective, then $\phi$ is a homeomorphism.
\end{la}
\begin{proof}
Since~$G$ and $H$ are $\sigma$-compact by \cite[1.8(a)]{Sie}
and locally compact, the Open Mapping Theorem
(e.g., \cite[Theorem 5.29]{HaR}) applies to~$\phi$.
\end{proof}
\begin{la}\label{nohoms}
Let $p\not=q$ be primes,
$G$ and $H$ be
contractible locally compact groups and $\phi\colon G\to H$
be a continuous group homomorphism.
Then $\phi=e$ in each of the following cases:
\begin{itemize}
\item[{\rm(a)}]
$G$ is connected and $H$ is a torsion group or a $p$-adic Lie group.
\item[{\rm(b)}]
$G$ is a $p$-adic Lie group and
$H$ is a torsion group or a $q$-adic Lie group.
\item[{\rm(c)}]
$G$ is a torsion group and $H$ is a $p$-adic Lie group.
\item[{\rm(d)}]
$G$ is locally pro-$p$ and $H$ is locally pro-$q$.
\item[{\rm(e)}]
$G$ is totally disconnected and~$H$ is connected.
\end{itemize}
\end{la}
\begin{proof}
(a) If $H$ is a torsion group, then $H$ is totally disconnected (see \ref{torsioncase}),
and so is every $p$-adic Lie group. As $\phi(G)$ is connected,
$\phi(G)=\{e\}$ follows.\vspace{1.5mm}

(b) If $H$ is a torsion group, we exploit
that each $g\in G$ is divisible (by Lemma~\ref{padcase}\,(b)).
Hence $\phi(g)$ is divisible and thus $\phi(g)=e$,
by Lemma~\ref{torsioncase}.
If $H$ is a $q$-adic Lie group, then the assertion is a special case of~(d).\vspace{1.5mm}

(c) Since $H$ is torsion-free as recalled in \ref{padcase} while $\phi(G)\sub H$ is a torsion group,
we must have $\phi(G)=\{e\}$.\vspace{1.5mm}

(d) Given $g\in G$, there exists a compact open subgroup~$U$ of~$G$
which is a pro-$p$-group and contains~$g$ (see Lemma~\ref{unionprop}).
Then $\phi(U)$ is a compact subgroup of~$H$ and hence a pro-$q$-group,
by Lemma~\ref{unionprop}. By Lemma~\ref{homspropq},
$\phi|_U\colon U\to \phi(U)$ is the constant map~$e$,
whence $\phi(g)=e$ in particular.\vspace{1.5mm}

(e) Each $g\in G$ is topologically periodic (see Lemma~\ref{tdtopper}),
whence $\phi(g)$ is topologically periodic and thus
$\phi(g)=e$ (see Lemma~\ref{conncase}).
\end{proof}
\begin{la}\label{varpq}
Let $p\not=q$ be prime numbers
and $\phi\colon G\to H$ be a continuous group homomorphism
from a contractible $p$-adic Lie group to a torsion-free
$q$-adic Lie group~$H$. Then $\phi=e$.
\end{la}
\begin{proof}
For each $g\in G$, we have $g^{p^n}\to e$ as $n\to\infty$,
whence $\phi(g)^{p^n}\to e$ and thus $\phi(g)=e$, by Lemma~\ref{torfreepad}\,(c).
\end{proof}
\begin{la}\label{complementsDE}
Let $\phi\colon G\to H$ be a continuous group homomorphism
between contractible locally compact groups $G$ and $H$.
Assume that
\[
G=G_e\times D\quad\mbox{and}\quad H=H_e\times E
\]
internally as topological groups for
closed normal subgroups $D\sub G$ and
$E\sub H$.
Then $\phi(G_e)\sub H_e$ and $\phi(D)\sub E$. If $\phi$ is surjective, then $\phi(G_e)=H_e$
and $\phi(D)=E$.
\end{la}
\begin{proof}
The assertion $\phi(G_e)\sub H_e$ is trivial.
Write $\phi|_D=(\phi_e,\theta)$ with $\phi_e\colon D\to H_e$
and $\theta\colon D\to E$.
Since $D\cong G/G_e$ is totally disconnected and contractible,
we have $\phi_e=e$ by Lemma~\ref{nohoms}\,(e) and thus
$\phi(D)\sub E$. Notably, $\phi(G)=\phi(G_e)\times \phi(D)$,
from which the final assertion follows.
\end{proof}
\begin{prop}\label{prethmA}
Let $(G,\alpha)$ and $(H,\beta)$ be locally compact contraction groups
and $\phi\colon G\to H$ be a continuous group homomorphism.
Assume that 
\begin{equation}\label{dcomG}
G=G_e\times G_{p_1}\times\cdots\times G_{p_n}\times \tor(G)
\end{equation}
internally as a topological group,
for certain prime numbers $p_1<\cdots<p_n$
and non-trivial, $\alpha$-stable closed normal subgroups $G_p\sub G$
which are $p$-adic Lie groups, for $p\in\{p_1,\ldots,p_n\}$.
Moreover, assume that
\begin{equation}\label{decoh}
H=H_e\times H_{q_1}\times\cdots\times H_{q_m}\times \tor(H)
\end{equation}
internally as a topological group,
for certain prime numbers $q_1<\cdots<q_m$
and non-trivial, closed normal subgroups $H_p\sub H$
which are $p$-adic Lie groups, for all $p\in\{q_1,\ldots,q_m\}$.
For primes $p$ such that $p\not\in\{q_1,\ldots, q_m\}$,
set $H_p:=\{e\}$.
Then the following holds:
\begin{itemize}
\item[\rm(a)]
$\phi(G_e)\sub H_e$ and $\phi(\tor(G))\sub \tor(H)$;
\item[\rm(b)]
$\phi(G_p)\sub H_p$ for all $p\in\{p_1,\ldots, p_n\}$;
\item[\rm(c)]
If $\phi$ is injective, then $\{p_1,\ldots, p_n\}\sub\{q_1,\ldots, q_m\}$;
\item[\rm(d)]
If $\phi$ is surjective, then $\phi(G_e)=H_e$, $\phi(\tor(G))=\tor(H)$,
$\{p_1,\ldots,p_n\}\supseteq \{q_1,\ldots,q_m\}$ and $\phi(G_p)=H_p$
for all $p\in\{q_1,\ldots,q_m\}$.
\end{itemize}
\end{prop}
\begin{proof}
(a) is obvious. To prove~(b),
note that $H_e\times H_{q_1}\times\cdots\times H_{q_m}\cong H/\!\tor(H)$ is torsion-free. So $H_q$ is torsion-free for all $q \in\!\{q_1,\ldots,q_m\}$.
Write
\[
\phi|_{G_p}=(\phi_e,\phi_1,\ldots,\phi_m,\theta)
\]
in terms of its components $\phi_e\colon G_p\to H_e$,
$\phi_j\colon G_p\to H_{q_j}$ for $j\in\{1,\ldots, m\}$,
and $\theta\colon G_p\to \tor(H)$
with respect to the decomposition~(\ref{decoh}).
By Lemma~\ref{nohoms} (e) and~(b), we have $\phi_e=e$ and $\theta=e$.
Moreover, $\phi_j=e$ for all $j\in\{1,\ldots, m\}$
such that $p \not= q_j$, by Lemma~\ref{varpq}.
If $p\not\in\{q_1,\ldots, q_m\}$, we deduce that $\phi(G_p)=\{e\}=H_p$.
If $p=q_j$ for some $j\in\{1,\ldots,m\}$, we deduce that
$\phi(G_p)\sub H_p$.

(c) is immediate from~(b).

(d) For $j\in\{1,\ldots,m\}$,
let $S_j:=\phi_j(G_{p_i})\sub H_{q_j}$ if there exists $i\in\{1,\ldots,n\}$
such that $p_i=q_j$; otherwise, we set $S_j:=\{e\}\sub H_{q_j}$.
Then
\begin{equation}\label{decoim}
\phi(G)=\phi(G_e)\times S_1\times\cdots\times S_m\times\phi(\tor(G)),
\end{equation}
where $\phi(G_e)\sub H_e$ and $\phi(\tor(G))\sub \tor(H)$.
If~$\phi$ is onto, then a comparison of (\ref{decoh}) and (\ref{decoim})
yields the conclusion of~(d).
\end{proof}
{\bf Proof of Theorem~A.}
(a) By \cite[Proposition~4.2]{Sie}, $G$ has a totally disconnected, $\alpha$-stable, closed
normal subgroup~$D$ such that $G=G_e\times D$ internally as a topological group.
If also $E$ is a closed normal subgroup of~$G$ such that $G=G_e\times E$
internally as a topological group, applying Lemma~\ref{complementsDE} to $\id_G\colon G_e\times D\to G_e\times E$
we obtain $D=E$. Again by Lemma~\ref{complementsDE}, $D$ is topologically fully invariant.\vspace{1.2mm}

(b) Write $G=G_e\times G_{\td}$ as in~(a). By Lemma~\ref{torisgp} and its proof,
$\tor(G)=\tor(G_{\td})$ is a closed subgroup of~$G$.
By \cite[Theorem~B]{GW}, there are primes $p_1<\cdots<p_n$
and $\alpha$-stable closed normal non-trivial subgroups $G_p$ of~$G$
for $p\in\{p_1,\ldots,p_n\}$ such that~(\ref{dcomG}) holds.
Now assume that $q_1<\cdots<q_m$ are primes and
\[
G=G_e\times H_{q_1}\times\cdots\times H_{q_m}\times \tor(G)
\]
internally as a topological group for non-trivial $p$-adic Lie groups
$H_p$ for $p\in \{q_1,\ldots, q_m\}$
which are closed normal subgroups
of~$G$.
Applying Proposition~\ref{prethmA} to $\phi:=\id_G$, we find that $n=m$,
$p_i=q_i$ for all $i\in\{1,\ldots,n\}$ and $G_p=H_p$ for all
$p\in\{p_1,\ldots,p_n\}$.
By Proposition~\ref{prethmA}, each $G_p$ is topologically fully invariant in~$G$.\vspace{1.3mm}

(c) will be established later in this section, after Proposition~\ref{beforeA}.\\[2.3mm]
More can be said concerning the $p$-adic Lie groups $G_p$ in Theorem~A.
\begin{prop}\label{uniq-thm}
Let $(G,\alpha)$, $p_1<\ldots<p_n$ and $G_p$ for $p\in\{p_1,\ldots, p_n\}$
be as in Theorem~{\rm A}.
Write $G_p:=\{e\}$
for primes $p\not\in\{p_1,\ldots,p_n\}$.
Then
\begin{equation}\label{identifGp}
G_p=\left\{x\in G\colon \mbox{$x$ is divisible and}\;\lim_{n\to\infty}x^{p^n}=e\right\}
\end{equation}
for each prime number~$p$.
\end{prop}
\begin{proof}
Let $\pi_e\colon G\to G_e$,
$\theta\colon G\to \tor(G)$
and
$\pi_p\colon G\to G_p$ (for $p\in\{p_1,\ldots, p_n\}$)
be the projections onto the components
of
\[
G=G_e\times G_{p_1}\times\cdots\times G_{p_n}\times \tor(G).
\]
For each prime number~$p$, the left-hand-side of (\ref{identifGp})
is a subset of the right-hand-side, by Lemma~\ref{padcase}\,(c).
To see the converse inclusion, let $x$ be an element of the right-hand-side.
Then $\theta(x)$ is a divisible element of $\tor(G)$ and hence $\theta(x)=e$,
by Lemma~\ref{torsioncase}.
Moreover, $\pi_e(x)=e$ by Lemma~\ref{conncase}\,(d).
Finally, $\pi_q(x)=e$ for each $q\in\{p_1,\ldots, p_n\}$ such that $q\not=p$,
by Lemma~\ref{torfreepad}. 
If $p\not\in\{p_1,\ldots,p_n\}$, we deduce that $x=e$, whence $x\in G_p$.
If $p\in\{p_1,\ldots,p_n\}$, we deduce that $x=\pi_p(x)\in G_p$.\vspace{4mm}
\end{proof}
{\bf Locally pro-nilpotent contraction groups}
\begin{prop}\label{all-about}
Let $(G,\alpha)$ be a locally compact contraction group
such that
$G_{\td}$ is locally pro-nilpotent.
Then $G$ has a largest $\alpha$-stable closed normal subgroup~$G_{(p)}$
which is locally pro-$p$,
for each prime number~$p$.
Moreover, $G_{(p)}\not=\{e\}$
only for finitely many primes, say for $p$ in the set $\{p_1,\ldots,p_m\}$
with $p_1<\cdots < p_m$,
and
\begin{equation}\label{decomax}
G=G_e\times G_{(p_1)}\times\cdots\times G_{(p_m)}
\end{equation}
internally as a topological group.
\end{prop}
\begin{proof}
Let $U\sub G_{\td}$ be a compact open subgroup which is pro-nilpotent.
Let $V\sub G$ be a compact open subgroup such that $\alpha(V)\sub V$ (see~\ref{inv-sub}).
Then $\alpha^n(V)\sub U$ for some $n\in\N$, by~\ref{compacontra}.
Thus $\alpha^n(V)$ is pro-nilpotent, and hence so is~$V$.
After replacing~$U$ with~$V$, we may assume that~$U$
is pro-nilpotent and $\alpha(U) \sub U$.
Being pro-nilpotent, $U$ has a unique $p$-Sylow subgroup~$S_p$
for each prime~$p$ (see \cite[Propositions 2.4.3\,(ii) and 2.2.2\,(d)]{Wil}), and
\begin{equation}\label{lngg}
U=\prod_{p\in\bP} S_p\vspace{-1.3mm}
\end{equation}
(see \cite[Proposition~2.4.3\,(iii)]{Wil}).
As~$\alpha$ is an automorphism, $\alpha(S_p)$ is the unique Sylow $p$-subgroup of~$\alpha(U)$.
Now $\alpha(S_p)$ is the intersection of a $p$-Sylow subgroup of~$U$ with $\alpha(U)\sub U$,
as follows from \cite[Proposition~2.2.2\,(c)]{Wil}. So
\begin{equation}\label{identsyl}
\alpha(S_p)=S_p\cap\alpha(U)\sub S_p.
\end{equation}
By \cite[Proposition~3.1]{GW}, $G_{(p)}:=\bigcup_{n\in\N_0}\alpha^{-n}(S_p)$
is a closed subgroup of~$G_{\td}$, and~$S_p$ is open in~$G_{(p)}$.
Now $\alpha^{-n}(S_p)$ is normal in
$\alpha^{-n}(U)$, for each $n\in\N_0$,
since $S_p$ is normal in~$U$. As a consequence,
the ascending union $G_{(p)}=\bigcup_{n\in\N_0}\alpha^{-n}(S_p)$
is normal in the ascending union $G_{\td}=\bigcup_{n\in\N_0}\alpha^{-n}(U)$
and hence in $G=G_e\times G_{\td}$.
By~(\ref{lngg}) and~(\ref{identsyl}), we have
\[
\alpha(U)=\prod_{p\in\bP}\alpha(S_p)=\prod_{p\in \bP}(\alpha(U)\cap S_p).\vspace{-1mm}
\]
As $\alpha(U)$ is open in~$U$, for all but finitely many~$p$
we must have $S_p=\alpha(U)\cap S_p$ (which is compact) and thus $S_p=\{e\}$
(as every compact contraction group is trivial). Let $p_1<\cdots < p_m$
be the primes~$p$ for which $S_p\not=\{e\}$.
Then
\[
G_{(p_1)}\cdots G_{(p_m)}=\bigcup_{n\in\N_0}\alpha^{-n}(U)=G_{\td},
\]
showing that the product map
\[
\pi\colon G_{(p_1)}\times\cdots\times G_{(p_m)}\to G_{\td},\;\; (g_1,\ldots, g_m)\mto g_1g_2\cdots g_m
\]
is surjective.
Now $S_{p_j}$
centralizes~$S_{p_k}$
for all $j\not=k$ in $\{1,\ldots,m\}$, whence the ascending union $G_{(p_j)}=
\bigcup_{n\in\N_0}\alpha^{-n}(S_j)$
centralizes the ascending union $G_{(p_k)}=\bigcup_{n\in\N_0}\alpha^{-n}(S_{p_k})$.
As a consequence, $\pi$ is a group homomorphism
and hence a morphism of contraction groups if we use
\[
\alpha|_{G_{(p_1})}\times\cdots\times
\alpha|_{G_{(p_m})}
\]
on the direct product.
Thus $\ker(\pi)$ is a contraction group.
As $\ker(\pi)\cap (S_{p_1}\times\cdots\times S_{p_m})=\{(e,\ldots,e)\}$,
the contraction group $\ker(\pi)$ is discrete and thus $\ker(\pi)=\{(e,\ldots,e)\}$.
Thus~$\pi$ is bijective and hence an isomorphism of contraction groups, by~\ref{nicemorph}.
As a consequence, also~(\ref{decomax}) holds.
For $k\in\{1,\ldots, m\}$, let $\pr_k\colon G\to G_{(p_k)}$ be the projection onto the factor $G_{(p_k)}$
in the decomposition~(\ref{decomax}). Moreover, let $\pr_e\colon G\to G_e$ be
projection onto the factor~$G_e$ in~(\ref{decomax}).
If $p$ is a prime and $N$ an $\alpha$-stable closed normal subgroup of~$G$
which is locally pro-$p$, then $\pr_k(N)=\{e\}$ for all
$k\in\{1,\ldots,m\}$ such that $p_k\not=p$;
moreover, $\pr_e(N)=\{e\}$ (see Lemma~\ref{nohoms} (d) and (e)).
If $p\not\in\{p_1,\ldots, p_m\}$,
this entails that $N=\{e\}$, whence $\{e\}$ is the largest $\alpha$-stable closed
normal subgroup of~$G$ which is locally pro-$p$.
If $p=p_k$ for some $k\in\{1,\ldots, m\}$,
the preceding shows that $N\sub G_{(p_k)}$,
whence $G_{(p_k)}$ is the largest $\alpha$-stable closed normal subgroup
of $G$ which is locally pro-$p$.
\end{proof}
\begin{rem}
We mention that,
in Proposition~\ref{all-about},
$\{p_1,\ldots, p_m\}$ is the set of all
prime divisors of $\Delta(\alpha^{-1}|_{G_{\td}})$.
Since
\[
\Delta(\alpha^{-1}|_{G_{\td}})=\Delta(\alpha^{-1}|_{G_{(p_1)}})\cdots
\Delta(\alpha^{-1}|_{G_{(p_m)}}),
\]
this assertion will follow from Lemma~\ref{mod-pro-p}.
\end{rem}
\begin{prop}\label{beforeA}
Let $(G,\alpha)$ be a totally disconnected locally compact
contraction group which is locally pro-nilpotent and a torsion group.
Let $p$ be a prime. Then
$\tor_p(G)$ coincides with $G_{(p)}$, the largest closed normal subgroup
of~$G$ which is $\alpha$-stable and locally pro-$p$.
\end{prop}
\begin{proof}
Let $p_1<\cdots<p_m$ and $S_p$ for $p\in\{p_1,\ldots,p_m\}$ be
as in Proposition~\ref{all-about} and its proof.
Then~$S_p$ is a pro-$p$-group and a torsion
group, whence $S_p\sub\tor_p(G)$ and hence also
\[
G_{(p)}=\bigcup_{n\in\N_0}\alpha^{-n}(S_p)\sub\tor_p(G).\vspace{-1mm}
\]
Since $G=G_{(p_1)}\times\cdots\times G_{(p_m)}$,
we deduce that $\tor_p(G)=\{e\}=G_{(p)}$ for all primes $p\not\in\{p_1,\ldots, p_m\}$.
Moreover, $\tor_p(G)=G_{(p)}$ if $p\in\{p_1,\ldots, p_m\}$.\vspace{4mm}
\end{proof}
{\bf Proof of Theorem~A, completed.}\\[2.3mm]
To establish~(c), recall from Lemma~\ref{torisgp}
that $H:=\tor(G)$ is a totally disconnected
$\alpha$-stable closed subgroup of~$G$. As~$H$ is locally pro-nilpotent
by hypothesis, Propositions \ref{all-about} and~\ref{beforeA}
show that $\tor_p(H)=\tor_p(G)$ is non-trivial
only for $p$ in a finite set of primes $q_1<\cdots< q_m$.
Moreover, $\tor_p(H)=H_{(p)}$ is locally pro-$p$ for each prime~$p$
and $\tor(H)=\tor_{q_1}(H)\times\cdots\times\tor_{q_m}(H)$.
Recalling \ref{fuinva}, all assertions are established.\vspace{2mm}\,\Punkt
\begin{rem}
We mention that the hypothesis in Theorem~A(c) is also\linebreak
necessary for its conclusion.\\[2.3mm]
To see this,
let $(G,\alpha)$ be a totally disconnected, locally compact contraction group
and $q_1<\cdots<q_m$ be prime numbers such that
$\tor_p(G)$ is a closed subgroup of~$G$ for
$p\in\{q_1,\ldots,q_m\}$ and
\[
G=\tor_{q_1}(G)\times\cdots\times \tor_{q_m}(G)
\]
internally as a topological group.
Then $\tor_p(G)$ is locally pro-$p$ for all $p\in\{q_1,\ldots,q_m\}$
(as each compact open subgroup $U_p$ of $\tor_p(G)$ is both pro-finite and
a $p$-group and hence a pro-$p$-group). Since $U_{q_1}\times\cdots\times U_{q_m}$
is a compact open subgroup of~$G$ and pro-nilpotent, we see that~$G$ is locally pro-nilpotent.
\end{rem}
\section{Existence of equivariant global sections}\label{secsec}
In this section, we prove Theorem~C.
If $(G,\alpha)$, $(H,\beta)$, and~$q$ are as in the theorem,
we have decompositions $G=G_e\times G_{\td}$
and $H=H_e\times H_{\td}$ as in Theorem~A.
Since $q$ is a group homomorphism such that
$q(G_e)=H_e$ and $q(G_{\td})=H_{\td}$ (see Lemma~\ref{complementsDE}),
we have
\[
q=q|_{G_e}\times q|_{G_{\td}}\colon G_e\times G_{\td}\to H_e\times H_{\td}.
\]
If we can find a continuous section $\sigma_1\colon H_e\to G_e$
for the surjective morphism
$q|_{G_e}\colon G_e\to H_e$ of contraction groups and a continuous
section $\sigma_2\colon H_{\td}\to G_{\td}$ for the surjective
morphism $q|_{G_{\td}}\colon G_{\td}\to H_{\td}$
such that $\alpha|_{G_e}\circ \sigma_1=\sigma_1\circ \beta|_{H_e}$
and $\alpha|_{G_{\td}}\circ\sigma_2=\sigma_2\circ \beta|_{H_{\td}}$,
then $\sigma:=\sigma_1\times\sigma_2$ will be a continuous section
for~$q$ such that $\alpha\circ\sigma=\sigma\circ\beta$.
Moreover, $\sigma(e)=\sigma(\beta^n(e))=\alpha^n(\sigma(e))$
for each $n\in\N$ and hence $\sigma(e)=\lim_{n\to\infty}\alpha^n(\sigma(e))=e$.
It therefore suffices to prove Theorem~C in two cases: a) both $G$ and $H$ are totally disconnected;
b) both $G$ and $H$ are connected. We start with case a);
case b) will be settled afterwards, in Lemma~\ref{sectionsconn}.\\[2.3mm]
{\bf Proof of Theorem~C when {\boldmath $G$ and $H$} are totally disconnected}\\[1mm]
Let $V\sub G$ be a compact open subgroup such that $\alpha(V)\sub V$ (see \ref{inv-sub}).
Then $\{\alpha^k(V)\colon k\in\N_0\}$
is a basis of identity neighbourhoods in~$G$ (see \ref{compacontra}).
By~\ref{nicemorph},
$q$~is an open map. Hence $U:=q(V)$ is a compact open subgroup of~$H$.
Since $\beta\circ q=q\circ \alpha$, we have $\beta(U)\sub U$.
By~\ref{compacontra}, $\{\beta^k(U)\colon k\in\N_0\}$ is a basis of identity neighbourhoods in~$H$.
Since $\beta(U)$ is open and~$U$ is compact,
the index $\ell:=[U:\beta(U)]$ is finite. We pick representatives $h_1,\ldots, h_\ell$
for the left cosets $h\beta(U)\in U/\beta(U)$ of $\beta(U)$ in~$U$, such that $h_1=e$.
We find $g_1,\ldots,g_\ell\in V$ such that $q(g_j)=h_j$ for all $j\in\{1,\ldots,\ell\}$, and $g_1=e$.
Let $m\in\Z$. For $h\in\beta^m(U)$,
we show by induction on $n\in \{m,m+1,\ldots\}$ that there are unique
$j_m,j_{m+1},\ldots,j_n\in\{1,\ldots,\ell\}$ such that
\begin{equation}\label{getreps}
h\beta^{k+1}(U)=\beta^m(h_{j_m})\cdots \beta^k(h_{j_k})\beta^{k+1}(U)
\end{equation}
for all
$k\in\{m,\ldots, n\}$.
If $n=m$, there is a unique $j_m\in \{1,\ldots, \ell\}$
with $h\beta^{m+1}(U)=\beta^m(h_{j_m})\beta^{m+1}(U)$
as $\beta^m(h_1),\ldots, \beta^m(h_\ell)$ are representatives for the left cosets
of $\beta^{m+1}(U)$ in $\beta^m(U)$. If $n>m$ and unique
$j_m,\ldots,j_{n-1}$ have been found such that (\ref{getreps})
holds for all $k\in\{m,\ldots, n-1\}$, then
\[
h=\beta^m(h_{j_m})\cdots \beta^{n-1}(h_{j_{n-1}})z
\]
for a unique $z\in \beta^n(U)$ and $z\beta^{n+1}(U)=\beta^n(h_{j_n})\beta^{n+1}(U)$ for
a unique $j_n\in\{1,\ldots,\ell\}$.
Then
\begin{equation}\label{givesuni}
h\beta^{n+1}(U)=\beta^m(h_{j_m})\cdots \beta^n(h_{j_n})\beta^{n+1}(U)
\end{equation}
and $j_n$ is uniquely determined by this property. In fact, if (\ref{givesuni})
also holds with $i\in\{1,\ldots,\ell\}$ in place of $j_n$, then we can multiply both identities with
$\beta^{n-1}(h_{j_{n-1}})^{-1}\cdots \beta^m(h_{j_m})^{-1}$ on the left. Comparing the right hand sides,
\[
\beta^n(h_{j_n})\beta^{n+1}(U)=\beta^n(h_i)\beta^{n+1}(U)
\]
follows and hence $i=j_n$.

For $n\in\{m,m+1,\ldots\}$, let $f_n(h):=\alpha^n(g_{j_n})$.
Then $f_n(h')=f_n(h)$ for all $h'\in\beta^m(U)$
such that $h'\in h \beta^{n+1}(U)$, showing that the function
\[
f_n\colon \beta^m(U)\to \alpha^n(V)
\]
is locally constant and thus continuous.
Define $s_n(h):=\alpha^m(g_{j_m})\cdots \alpha^n(g_{j_n})$ for $n\in\{m,m+1,\ldots\}$.
By Lemma~\ref{egsum}, the limit
\[
\sigma(h):=\lim_{n\to\infty}s_n(h)
\]
exists for all $h\in \beta^m(U)$, and the function $\beta^m(U)\to G$,
$h\mto\sigma(h)$ is continuous. Given $h\in H$, there is $m\in\Z$ such that $h\in \beta^m(U)$ and 
our definition of $\sigma(h)$ does not depend on the choice of $m$. We therefore obtain a well-defined
function $\sigma\colon H\to G$, $h\mto \sigma(h)$ which is continuous
as each of the subsets $\beta^m(U)$ is open in~$H$ and $H=\bigcup_{m\in\Z}\beta^m(U)$.
Since
\[
q(s_n(h))=\beta^m(q(g_{j_m}))\cdots  \beta^n(q(g_{j_n}))
=\beta^m(h_{j_m})\cdots\beta^n(h_{j_n})\in h\beta^{n+1}(U)
\]
for $m\in\Z$, $h\in \beta^m(U)$ and $n\in\{m,m+1,\ldots\}$,
we see that
\[
q(\sigma(h))=\lim_{n\to\infty}q(s_n(h))=h.
\]
If $h\in\beta^m(U)$, then $\beta(h)\in\beta^{m+1}(U)$ holds
and the indices $k_i$ determined by $\beta(h)$ are $k_i=j_{i-1}$ in terms
of those determined as above by~$h$, for all\linebreak
$i\in\{m+1,m+2,\ldots\}$.
Hence
\begin{eqnarray*}
\sigma(\beta(h))
&=&\lim_{n\to\infty}\alpha^{m+1}(g_{k_{m+1}})\cdots \alpha^n(g_{k_n})\\
&=&\alpha\left(\lim_{n\to\infty}\alpha^m(g_{j_m})\cdots \alpha^{n-1}(g_{j_{n-1}})\right)
\,=\, \alpha(\sigma(h))
\end{eqnarray*}
and thus $\sigma\circ\beta=\alpha\circ\sigma$.\,\Punkt
\begin{la}\label{insidenil}
Let $\K=\R$ or $\K$ be a field $\Q_p$ of $p$-adic numbers.
Let $\cg$ and $\ch$ be finite-dimensional nilpotent Lie algebras
over~$\K$; write $*$ both for the Baker-Campbell-Hausdorff
multiplication $\cg\times\cg\to\cg$ on~$\cg$
and the one on~$\ch$.
If $\phi\colon (\cg,*)\to(\ch,*)$ is a continuous group
homomorphism, then $\phi\colon \cg\to\ch$ is a homomorphism
of Lie algebras. Thus $\ker(\phi)$ is an ideal of the Lie algebra
$\cg$ and $\phi(\cg)$ is a Lie subalgebra of~$\ch$.
Notably, $\phi(\cg)$ is closed in~$\ch$ and the corestriction
$\phi|^{\phi(\cg)}\colon\cg\to\phi(\cg)$ is an open map.
\end{la}
\begin{proof}
For each $x\in \cg$ and $n\in\N$, we have $\phi(nx)=\phi(x^n)=\phi(x)^n=n\phi(x)$,
using~(\ref{powermult}). Hence $\phi(zx)=z\phi(x)$ for all rational numbers
$z$ (by unique divisibility) and thus $\phi(zx)=z\phi(x)$
for all $z\in\K$, by continuity of~$\phi$ and density of~$\Q$ in~$\K$.
As a consequence,
\[
L(\phi)(x)=\phi'(0)(x)=\frac{d}{dz}\Big|_{z=0}\phi(zx)=\phi(x).
\]
Thus $\phi=L(\phi)$ is a Lie algebra homomorphism. All other assertions are now immediate.
\end{proof}
We refer to $C^\infty$-maps between
real $C^\infty$-manifolds as \emph{smooth} maps,
as usual. If $f\colon M\to N$ is a mapping between such manifolds
(of dimensions $m$ and~$n$, respectively),
we say that~$f$ is \emph{differentiable} at a point $x\in M$
if $\psi\circ f\circ \phi^{-1}$ is differentiable at~$\phi(x)$
for each $C^\infty$-diffeomorphism~$\phi$ from an open neighbourhood of~$x$ in~$M$
onto an open subset of~$\R^m$ and each $C^\infty$-diffeomorphism~$\psi$
from an open neighbourhood of $f(x)$ in~$N$ onto an open subset of~$\R^n$.
By the Chain Rule, it suffices to check differentiability
at $\phi(x)$ for one pair $(\phi,\psi)$.\\[2.3mm]
If $(E,\|.\|)$ is a normed space over~$\R$, then we write for
$x\in E$ and $r>0$
\[
B^E_r(x):=\{y\in E\colon \|y-x\|<r\}\;\,\mbox{and}\;\,
\wb{B}^E_r(x):=\{y\in E\colon \|y-x\|\leq r\}.
\]
\begin{la}\label{sectionsconn}
Let $(G,\alpha)$ and $(H,\beta)$
be connected locally compact contraction groups
and $q \colon G\to H$ be a surjective morphism of contraction groups.
Then there exists a continuous map $\sigma\colon H\to G$
which is smooth on $H\setminus\{e\}$
such that $q \circ \sigma=\id_H$ and $\alpha\circ\sigma=\sigma\circ \beta$.
\end{la}
\begin{proof}
By Lemmas~\ref{conncase} and~\ref{insidenil},
we may assume that $G=(\cg,*)$
and $H=(\ch,*)$ for nilpotent Lie algebras~$\cg$ and~$\ch$,
endowed with the BCH-multiplication;
moreover, we may assume
that $\alpha$ and~$\beta$ are contractive Lie algebra automorphisms,
and~$q$ is a surjective Lie algebra homomorphism.
There exists a linear map $\tau\colon \ch\to\cg$
such that $q\circ \tau=\id_\ch$.
Notably, $\tau$ is a smooth section for~$q$.
However, $\tau$ need not be equivariant.
To construct an equivariant section, we proceed as follows.
By Lemma~\ref{reallin},
there exists a norm~$\|.\|$ on~$\ch$ with respect
to which $\theta:=\|\beta\|_{\op}<1$.
Pick $r\in\,]\theta,1[$ and
set
\[
U:=B^\ch_1(0),\quad W:= B^\ch_r(0).
\]
Then $\wb{W}=\wb{B}^\ch_r(0)\sub U$ holds, $\wb{U}=\wb{B}^\ch_1(0)$ and
\[
\beta(\wb{U})\sub\wb{B}^\ch_\theta(0)\sub B^\ch_r(0)=W.
\]
Notably, $\beta(U)\sub U$ and
\begin{equation}\label{iteralso}
\beta^k(U)\sub W\quad\mbox{for all $k\in\N$.}
\end{equation}
Moreover,
\begin{equation}\label{thuscove}
U\setminus \{0\}=\bigcup_{n\in\N_0}\beta^n(U\setminus \beta(U))
=\bigcup_{n\in\N_0}\beta^n(U\setminus \beta(\wb{W}))
\end{equation}
and $\ch\setminus\{0\}=\bigcup_{n\in\Z}\beta^n(U\setminus\beta(\wb{W}))$.
For each open cover of the finite-dimensional smooth manifold~$\ch$,
there exists a smooth partition of unity subordinate to it (cf.\ \cite[Corollary~III.2.4]{Lan}).
As the sets $W$ and $\ch\setminus \beta(\wb{U})$
form an open cover of~$\ch$,
we therefore find smooth maps $g\colon\ch\to\R$
and $h\colon\ch\to\R$ with image in $[0,1]$
such that $h+g=1$ and $\Supp(g)\sub W$, $\Supp(h)\sub
\ch\setminus\beta(\wb{U})$.
Note that also $\tau_1:=\alpha\circ\tau\circ\beta^{-1}$
is a section for~$q$. We define
\[
\sigma_0\colon
\ch \to \cg,\quad
x\mto
\tau(x)*(g(x)(\tau(x)^{-1}*\tau_1(x))).
\]
Then $\sigma_0$ is smooth and a section for~$q$
(as $\ker(q)$ is a vector subspace of~$\cg$ and $\tau(x)^{-1}*\tau_1(x)\in\ker(q)$).
Since $g(x)=0$ for $x\in\ch\setminus W$, we have
\begin{equation}\label{fipa}
\sigma_0(x)=\tau(x)\quad\mbox{for all $\, x\in \ch\setminus W$.}
\end{equation}
Since $g(x)=1$ for all $x\in\beta(\wb{U})$, we have
\begin{equation}\label{secpa}
\sigma_0(x)=\tau_1(x)=\alpha(\tau(\beta^{-1}(x)))\quad\mbox{for all $\,x\in\beta(\wb{U})$.}
\end{equation}
We now define a map $\sigma\colon \ch\to\cg$ via $\sigma(0):=0$ and
\[
\sigma(x):=\alpha^n(\sigma_0(\beta^{-n}(x)))\quad\mbox{if $\,n\in\Z\,$ and $x\in \beta^n(U\setminus\beta(\wb{W}))$.}
\]
To see that $\sigma$ is well-defined, let $x\in\ch$ and
$n<m$ be integers such that
\[
x\in \beta^n(U\setminus \beta(\wb{W}))\cap\beta^m(U\setminus \beta(\wb{W})).
\]
Then $U\setminus\beta(\wb{W})\cap \beta^{m-n}(U\setminus \beta(\wb{W}))\not=\emptyset$,
whence $\beta^{m-n}(U)$ cannot be a subset of~$\beta(W)$ and thus $m-n-1\leq 0$,
by~(\ref{iteralso}). Hence $m=n+1$.
Now $y:=\beta^{-n-1}(x)=\beta^{-m}(x)\in U\setminus\beta(\wb{W})$ and
\begin{equation}\label{chk1}
\beta(y)\in (U\setminus\beta(\wb{W}))\cap \underbrace{\beta(U\setminus \beta(\wb{W}))}_{\sub \, \beta(U)\sub \, U}
\sub
\beta(U)\setminus \beta(\wb{W})=\beta(U\setminus \wb{W}),
\end{equation}
whence
\begin{equation}\label{chk2}
y \in U\setminus \wb{W}.
\end{equation}
Then
\begin{eqnarray}
\alpha^n(\sigma_0(\beta^{-n}(x)))
&=& \alpha^n(\sigma_0(\beta(y)))=\alpha^{n+1}(\tau(y))
=\alpha^{n+1}(\sigma_0(y))\notag\\
&=&\alpha^{n+1}(\sigma_0(\beta^{-n-1}(x)))
=\alpha^m(\sigma_0(\alpha^{-m}(x))).\label{henceok}
\end{eqnarray}
In fact, the second equality holds by~(\ref{secpa}), applied to $\beta(y)$ in place of~$x$,
which is possible by~(\ref{chk1}).
The third equality holds by~(\ref{fipa}), which is applicable by~(\ref{chk2}).
In view of~(\ref{henceok}), $\sigma$ is well-defined.\\[2.3mm]
By definition, $\sigma$ coincides with the smooth section
$\alpha^n \circ \sigma_0\circ\beta^{-n}$ of~$q$ on the open set
$\beta^n(U\setminus\beta(\wb{W}))$ for each $n\in\Z$.
As the latter sets cover $\ch\setminus\{0\}$ (cf.\ (\ref{thuscove})),
we see that $\sigma|_{\ch\setminus \{0\}}$ is smooth
and a section for~$q$. Thus~$\sigma$ is a section for~$q$,
since also $q(\sigma(0))=q(0)=0$.\\[2.3mm]
To see that
$\sigma$ is continuous, it only remains to check its continuity at~$0$.
To this end, let $V\sub \cg$ be a $0$-neighbourhood.
Since $\sigma_0(\wb{U})$ is compact and $\alpha$ is contractive,
there exists $n_0\in\N_0$ such that
\[
\alpha^n(\sigma_0(\wb{U}))\sub V\quad\mbox{for all $\,n\geq n_0$.}
\]
As a consequence, $\sigma(x)\in V$ for all
$x\in\bigcup_{n\geq n_0}\beta^n(U\setminus\beta(\wb{W}))=
\beta^n(U)\setminus\{0\}$ (cf.\ (\ref{thuscove})).
Since $\sigma(0)=0\in V$, we deduce that $\sigma(\beta^n(U))\sub V$.
Thus $\sigma$ is continuous at~$0$.\\[2.3mm]
The condition $\alpha(\sigma(x))=\sigma(\beta(x))$ is trivial if $x=0$.
If $x\in\ch\setminus\{0\}$, then $x\in\beta^n(U\setminus\beta(\wb{W}))$
for some $n\in\Z$ (cf.\ (\ref{thuscove})).
Then $\beta(x)\in\beta^{n+1}(U\setminus\beta(\wb{W}))$
and we get
\[
\sigma(\beta(x))=\alpha^{n+1}(\sigma_0(\beta^{-n-1}(\beta(x))))=
\alpha(\alpha^n(\sigma_0(\beta^{-n}(x))))=\alpha(\sigma(x)),
\]
using the piecewise definition of~$\sigma$ for the first equality
and the last.
\end{proof}
Let us briefly discuss differentiability of equivariant sections at~$e$.
\begin{la}\label{discusssmooth}
In the setting of Lemma~{\rm\ref{sectionsconn}},
the following are equivalent:
\begin{itemize}
\item[\rm(a)]
$\sigma$ can be chosen as a smooth map;
\item[\rm(b)]
$\sigma$ can be chosen differentiable at~$e$;
\item[\rm(c)]
There exists a linear map $\tau\colon L(H)\to L(G)$
such that $L(q)\circ \tau=\id_{L(H)}$
and $L(\alpha)\circ \tau=\tau\circ L(\beta)$.
\end{itemize}
\end{la}
\begin{proof}
The implication ``(a)$\impl$(b)'' is trivial and also ``(c)$\impl$(a)'', once we identify
$(G,H,q,\alpha,\beta)$ with $(L(G),L(H),L(q),L(\alpha),L(\beta))$
as explained in the preceding proof.
To see that (b) implies (c), consider a continuous section $\tau\colon L(H)\to L(G)$
for $L(q)$ which is differentiable at~$0$
and satisfies
\begin{equation}\label{equilin}
L(\alpha)\circ\tau=\tau\circ L(\beta).
\end{equation}
Then $L(q)\circ\tau=\id_{L(H)}$ has the derivative
\[
\id_{L(H)}=L(q)\circ \tau'(0)
\]
at~$0$, by the Chain Rule. Thus $\tau'(0)\colon L(H)\to L(G)$
is a linear section for~$L(q)$.
Differentiating (\ref{equilin}) at~$0$, we get $L(\alpha)\circ\tau'(0)=\tau'(0)\circ L(\beta)$.
\end{proof}
We mention that $\sigma$ in Lemma~\ref{sectionsconn}
cannot always be chosen as a smooth map.
\begin{example}
Recall from the introduction that $\Omega_\infty$ is the set of monic irreducible real polynomials $f$ such that all eigenvalues of~$f$ have absolute value less than~$1$, and that $\alpha_{f^n}$ is the automorphism of $E_{f^n} = \R[X]/{f^n}\R[X]$ induced by multiplication by~$X$, for $n\in\N$. 
For $f\in\Omega_\infty$,
the mapping\linebreak
$q\colon E_{f^2}\to E_f$,
$g+f^2\R[X]\mto g+f\R[X]$
is a surjective morphism of contraction groups from $(E_{f^2},\alpha_{f^2})$
to $(E_f, \alpha_f)$ which does not admit an equivariant linear
section (as the latter would be an $\R[X]$-module homomorphism
and thus $E_{f^2}\cong (E_f)^2$, which is absurd).
By Lemma~\ref{discusssmooth},
$q$ cannot admit a smooth equivariant section.
\end{example}
\section{More on homomorphisms, series and\\
\hspace*{.4mm}simple factors}\label{newsec}
In this section, we first provide further results concerning
continuous group homomorphisms between locally compact
contraction groups, including a proof of the statements of Theorem~G.
These results ensure a certain ``closed product property'',
which is essential for our proof of a Jordan-H\"{o}lder
Theorem. After a brief review of facts concerning
contractive automorphisms of real vector spaces,
we then establish the classification described in
Theorem~F.
The section closes with preparatory results for Section~\ref{secab}.\\[7mm]
{\bf Results related to Theorem~G}
\begin{prop}\label{nicehomtf}
Let $\phi\colon G\to H$ be a continuous group homomorphism
between contractible locally compact groups~$G$ and~$H$.
Then $\phi(G)$ is closed in~$H$ if and only if $\phi(\tor(G))$ is closed.
In this case,
$\phi|^{\phi(G)}\colon G\to\phi(G)$
is an open map.
\end{prop}
\begin{proof}
Let $\alpha$ and $\beta$ be contractive automorphisms
for~$G$ and $H$, respectively.
Recall from Theorem~A that $\tor(G)$ is a closed subgroup of~$G$
and $G=G_e\times G_{p_1}\times\cdots\times G_{p_n}\times \tor(G)$
for certain primes $p_1<\cdots<p_n$ and $p$-adic Lie groups~$G_p$
which are $\alpha$-stable closed normal subgroups of~$G$.
Set $G_p:=\{e\}$ for primes $p\not\in\{p_1,\ldots, p_n\}$.
Likewise, $\tor(H)$ is closed in~$H$ and
\begin{equation}\label{decoh2}
H=H_e\times H_{q_1}\times\cdots\times H_{q_m}\times \tor(H)
\end{equation}
for certain primes $q_1<\cdots<q_m$ and $p$-adic Lie groups~$H_p$
which are $\beta$-stable closed normal subgroups of~$H$.
By Proposition~\ref{prethmA}, we have $\phi(\tor(G))\sub\tor(H)$ and
\[
\phi(G)=\phi(G_e)\times \phi(G_{q_1})\times\cdots\times\phi(G_{q_m})\times\phi(\tor(G)),
\]
where $\phi(G_e)$ is closed in $H_e$ and $\phi(G_p)$ is closed in $H_p$
for all $p\in\{q_1,\ldots, q_m\}$, by Lemmas~\ref{conncase}, \ref{padcase}
and~\ref{insidenil}. In view of (\ref{decoh2}), the desired equivalence follows.
Now $G$ is $\sigma$-compact, whence also $\phi(G)$ is $\sigma$-compact.
If $\phi(G)$ is closed and hence locally compact,
the Open Mapping Theorem \cite[Theorem~5.29]{HaR} shows that $\phi|^{\phi(G)}$
is an open map.
\end{proof}
For example, $\phi(G)$ is closed for
each continuous group homomorphism\linebreak
$\phi\colon G\to H$
between contractible locally compact groups such that $\tor(G)\sub\ker(\phi)$
(which is the case if $G$ or $H$ is torsion-free).\\[2.3mm]
The next example shows that $\phi(G)$ need not be closed in Proposition~\ref{nicehomtf}.
\begin{example}
Consider the additive group $A:=\bF_p(\!(t)\!)=\bF_p^{(-\N)}\times\bF^{\N_0}$
with the right shift $\alpha\colon A\to A$, $f\mto tf$.
Then the map
\[
\phi\colon A\to A,\quad (a_n)_{n\in\Z}\mto\sum_{n=1}^\infty a_{-2n}t^{-n}
+\sum_{n=0}^\infty a_{-2n-1}t^n\vspace{-1mm}
\]
is a continuous group homomorphism. The image of~$\phi$ is the
set $\bF_p^{(\Z)}$ of all finitely supported sequences in~$A$.
Thus $\phi(A)$ is a dense proper subgroup of~$A$,
and so~$\phi(A)$ is not closed in~$A$.
The co-restriction $\phi|^{\phi(A)}\colon A\to\phi(A)$
is not an open map, as otherwise $\phi(A)$ would be isomorphic to
the locally compact group $A/\ker(\phi)$, whence $\phi(A)$
would be complete
and hence closed in~$A$ (a contradiction).
\end{example}
\begin{cor}\label{precloprod}
If $\phi\colon (G,\alpha)\to (H,\alpha)$ is a morphism of
locally compact contraction groups, then $\phi(G)$ is a closed
subgroup of~$H$ and $\phi|^{\phi(G)}$ is an open map.
\end{cor}
\begin{proof}
We know from Lemma~\ref{torisgp}
that $\tor(G)$ is an $\alpha$-stable closed
subgroup of~$G$ and totally disconnected. Likewise,
$\tor(H)$ is a closed subgroup of~$H$ and totally disconnected.
Hence $\phi(\tor(G))$ is closed in $\tor(H)$ and hence in~$H$,
by~\ref{nicemorph}. The assertions now follow from Proposition~\ref{nicehomtf}.
\end{proof}
In particular, locally compact contraction groups always have
the following ``closed product propertiy'', which previously
had only been recorded for totally disconnected~$G$ (see \cite[Corollary~3.2]{GW}).
\begin{cor}\label{cloprod}
Let $(G,\alpha)$ be a locally compact contraction group,
$N$ be an $\alpha$-stable closed subgroup of~$G$
and $H$ be an $\alpha$-stable closed subgroup of~$G$
which normalizes~$N$. Then $NH$ is a closed $\alpha$-stable
subgroup of~$G$.
\end{cor}
\begin{proof}
Let $H$ act on $N$ via $h.n:=hnh^{-1}$ for $h\in H$ and $n\in N$.
The corresponding semidirect product $N\rtimes H$ is locally compact
group and the map $\alpha|_N\times\alpha|_H$ a contractive automorphism thereof.
As the product map
\[
\pi\colon N\rtimes H\to G,\quad (n,h)\mto nh
\]
is a morphism of contraction groups, we deduce with Corollary~\ref{precloprod}
that $NH=\pi(N\rtimes H)$ is closed in~$G$.\vspace{5mm}
\end{proof}
{\bf Composition series and the Jordan-H\"{o}lder Theorem}
\begin{numba}
A series (\ref{serial}) of contraction groups
is called a \emph{refinement}
of a series
\begin{equation}\label{series2}
\{e\}=H_0\lhd H_1\lhd\cdots\lhd H_m=G
\end{equation}
of contraction groups
if $\{H_0,\ldots, H_m\}\sub \{G_0,\ldots, G_n\}$.
If $m=n$ and there is a permutation
$\pi$ of $\{1,\ldots, n\}$ such that
\[
G_j/G_{j-1}\cong H_{\pi(j)}/H_{\pi(j)-1}
\]
as contraction groups
for all $j\in\{1,\ldots, n\}$, then~(\ref{serial})
and~(\ref{series2}) are called \emph{isomorphic}
series of contraction groups.
\end{numba}
\begin{la}\label{conncomps}
Let $(G,\alpha)$ be a locally compact contraction group.
If $G$ is connected, then $n\leq \dim_\R L(G)$
for each series $\{e\}=G_0\lhd G_1\lhd\cdots\lhd G_n=G$
of contraction groups for $(G,\alpha)$
such that $G_{j-1}$ is a proper subgroup of $G_j$ for all $n\in\{1,\ldots, n\}$.
In particular, $(G,\alpha)$ admits a composition series
of contraction groups.
\end{la}
\begin{proof}
We may assume that $G=(\cg,*)$ for a nilpotent real Lie algebra~$\cg$
and that $\alpha$ is a contractive Lie algebra automorphism.
Lemma~\ref{insidenil} implies that $G_j$ is a Lie subalgebra of~$\cg$ for each $j\in \{1,\ldots, n\}$,
and $\dim_\R(G_{j-1})< \dim_\R(G_j)$.
The first assertion follows from this and also the second,
as each properly ascending series (\ref{serial}) with maximal~$n$
will be a composition series.
\end{proof}
The following result generalizes the case of totally disconnected groups
discussed in \cite[Theorem~3.3]{GW}.
\begin{thm}\label{JH}
For each locally compact contraction group~$(G,\alpha)$, we have:
\begin{itemize}
\item[\rm(a)]
There exists a composition series $\{e\}=G_0\lhd G_1\lhd\cdots\lhd G_n=G$
of contraction groups for $(G,\alpha)$.
\item[\rm(b)]
Zassenhaus Lemma:
Any two series of contraction groups for $(G,\alpha)$
admit refinements which are isomorphic as series of contraction groups.
\item[\rm(c)]
Jordan-H\"{o}lder Theorem:
Every series of contraction groups for $(G,\alpha)$ can be refined to
a composition series of contraction groups. Moreover,
any two composition series of contraction groups for $(G,\alpha)$
are isomorphic as series of contraction groups.
\end{itemize}
\end{thm}
\begin{proof}
(a) By \cite[Proposition~4.2]{Sie}, $G=G_e\times G_{\td}$ internally
for a totally disconnected, $\alpha$-stable closed normal subgroup
$G_{\td}$ of~$G$. By Lemma~\ref{conncomps}, $G_e$ has a composition series
$\{e\}=G_0\lhd G_1\lhd\cdots\lhd G_n=G_e$
of contraction groups. By \cite[Theorem~3.3]{GW}, $G_{\td}$ has a composition series
$\{e\}=H_0\lhd H_1\cdots\lhd H_m=G_{\td}$
of contraction groups.
Then
\[
\{e\}=G_0\lhd G_1\lhd\cdots\lhd G_n=G_e\times H_0\lhd
G_e\times H_1\lhd\cdots\lhd G_e\times H_m=G
\]
is a composition series of contraction groups for $(G,\alpha)$.

(b) and (c) holds by \cite[Proposition~2.4]{GW},
in view of (a) and the ``closed product property'' established in Corollary~\ref{cloprod}.\vspace{4mm}
\end{proof}
{\bf Contractive automorphisms of real and complex vector spaces}\\[2.3mm]
We briefly review some facts concerning contractive automorphisms
of real (or complex) finite-dimensional vector spaces.
\begin{la}\label{complxlin}
Let $\alpha\colon E\to E$
be a $\C$-linear automorphism of a finite-dimensional complex vector space~$E$.
Then the following conditions are equivalent:
\begin{itemize}
\item[\rm(a)]
$\alpha$ is contractive;
\item[\rm(b)]
$|\lambda|<1$ for all
eigenvalues $\lambda$ of~$\alpha$ in~$\C$;
\item[\rm(c)]
There exists a norm $\|.\|$ on~$E$
such that
$\alpha$ has operator norm $\|\alpha\|_{\op}<1$
with respect to~$\|.\|$.
\end{itemize}
\end{la}
\begin{proof}
The implication ``(c)$\impl$(a)'' is trivial and also ``(a)$\impl$(b)'' (since
$\neg$(b)$\,\impl\!\neg$(a)). To complete the proof, it suffices to
prove the implication ``(b)$\impl$(c)'' in the case when matrix describing~$\alpha$
with respect to a suitable basis of~$E$ consists of a single Jordan block
(using the convention that ones appear below the diagonal).\footnote{In general,
$E$ is a direct sum $\bigoplus_{j=1}^nE_j$ of $\alpha$-stable vector subspaces
$E_j$ with this property.
Then $\|\alpha\|_{\op}<1$ with respect to the norm $\|.\|$ on~$E$
given by $\|x_1+\cdots+ x_n\|:=\max\{\|x_1\|_1,\ldots,\|x_n\|_n\}$
for $x_j\in E_j$, where $\|.\|_j$ is a norm on~$E_j$ with respect to which
$\|\alpha|_{E_j}\|_{\op}<1$.}
So, let $\lambda\in\C$, $k\in\N$ and $0\not=v\in E$ such that $(\alpha-\lambda\id_E)^k(v)=0$
and the vectors $v_j:=(\alpha-\lambda \id_E)^j(v)$ form a basis for~$E$
for $j\in\{0,1,\ldots, k-1\}$.
For $\theta>0$, also the vectors $\theta^{-j}v_j$ form a basis for~$E$
for $j\in\{0,1,\ldots,k-1\}$. Since
\[
\alpha\left(\sum_{j=0}^{k-1} x_j\theta^{-j}v_j\right)=\lambda x_0\theta^0 v_0+
\sum_{j=1}^{k-1}(\lambda x_j+\theta x_{j-1})\theta^{-j} v_j
\]
for $x_0,\ldots,x_{k-1}\in\C$, we have $\|\alpha\|_{\op}\leq|\lambda|+\theta$
with respect to the norm on~$E$ given by $\|\sum_{j=0}^{k-1}x_j\theta^{-j}v_j\|:=
\max\{|x_0|,\ldots,|x_{k-1}|\}$. Thus $\|\alpha\|_{\op}<1$
if we choose $\theta\in\;]0,1-|\lambda|[$.
\end{proof}
\begin{la}\label{reallin}
Let $\alpha\colon E\to E$ be an $\R$-linear automorphism of
a finite-dimensional real vector space~$E$.
The following conditions are equivalent:
\begin{itemize}
\item[\rm(a)]
$\alpha$ is contractive;
\item[\rm(b)]
The corresponding complex linear map $\alpha_\C\colon E_\C\to E_\C$ is contractive,
where $E_\C:=E\otimes_\R\C$ and $\alpha_\C:=\alpha\otimes\id_\C$;
\item[\rm(c)]
There exists a norm $\|.\|$ on~$E$
such that
$\alpha$ has operator norm $\|\alpha\|_{\op}<1$
with respect to~$\|.\|$.
\end{itemize}
\end{la}
\begin{proof}
The implication ``(c)$\impl$(a)'' is trivial.

(a)$\impl$(b): Considering $E_\C$ as a real vector space, we have $E_\C=E\oplus iE \cong E\times E$
and $\alpha_\C$ corresponds to $\alpha\times \alpha$, which is contractive.

(b)$\impl(c)$: If (b) holds, then there exists a norm $\|.\|$ on $E_\C$ with respect to which
$\|\alpha_\C\|_{\op}<1$. Identify $x\in E$ with $x\otimes 1\in E_\C$.
Then $\|\alpha\|_{\op}<1$ with respect to the restriction of~$\|.\|$ to~$E$.\vspace{4mm}
\end{proof}
{\bf Classification of the simple locally compact contraction groups}\\[2.3mm]
A classification of the simple totally disconnected
contraction groups was obtained in \cite[Theorem~A and Proposition~6.3]{GW}.
The following theorem completes the picture for general
locally compact contraction groups.
Recall from the introduction that
$\Omega_\infty$ is the set of all monic irreducible polynomials
$f$ over~$\R$ such that $|\lambda|<1$ for all zeros $\lambda$ of~$f$
in~$\C$.
Given a polynomial $f\in \Omega_\infty$, we endow
$E_f:=\R[X]/f\R[X]$
with the $\R$-linear automorphism~$\alpha_f$ given by multiplication of
representatives by~$X$.\\[2.3mm]
We assume that the reader is familiar with the theory of
finitely generated modules over a principal ideal domain
(like $\R[X]$) and its relations to normal forms
for endomorphisms of finite-dimensional vector spaces
(see, e.g., \cite[Chapter~3]{Jac}). We now establish the classification
described in Theorem~F.\\[2.3mm]
{\bf Proof of Theorem~F.}
As $G=G_e\times G_{\td}$ internally and $(G,\alpha)$ is simple,
we must have $G_e=\{e\}$ or $G_{\td}=\{e\}$,
establishing the first assertion.
If $G$ is connected,
then $G$ is a nilpotent real Lie group (see Lemma~\ref{conncase}),
whence its centre $Z(G)$ is a non-trivial $\alpha$-stable closed normal
subgroup of~$G$ and thus $G=Z(G)$,
as $(G,\alpha)$ is assumed simple.
Hence~$G$ is abelian. As a consequence, the Lie algebra
$L(G)$ is abelian, i.e., $[x,y]=0$ for all $x,y\in L(G)$
(cf.\ Proposition~7 in \cite[Chapter III, \S9.3]{Bou}).
The BCH-multiplication on $L(G)$ therefore is the addition map,
\[
x*y=x+y\;\;\mbox{for all $x,y\in L(G)$.}
\] 
Thus $(G,\alpha)$ is isomorphic to $((L(G),+),L(\alpha))$
as a contraction group (see Lemma~\ref{conncase}).
We may hence assume that~$G$ is a finite-dimensional
real vector space~$E$ and $\alpha\colon E\to E$
a contractive linear automorphism. If $p\in\R[X]$ is the characteristic polynomial
of~$\alpha$, then $|\lambda|<1$ for all roots of~$p$ in~$\C$, by
Lemmas~\ref{reallin} and \ref{complxlin}.
Consider~$E$ as an $\R[X]$-module via $X.v=\alpha(v)$
for $v\in E$. Since $(E,\alpha)$ is a simple contraction group,
$p$ coincides with the minimal polynomial of~$\alpha$
and $\R[X]/p\R[X]\cong E$
is a simple $\R[X]$-module, whence~$p$ is irreducible. Thus $p\in\Omega_\infty$
and $(E,\alpha) \cong \R[X]/p\R[X] = (E_p,\alpha_p)$.\\[2.3mm]
Conversely, $(E_f,\alpha_f)$ is a contraction group for each
$f\in \Omega_\infty$, by Lemmas~\ref{reallin} and~\ref{complxlin}.
Since each closed $\alpha$-stable subgroup of~$E_f$ is a vector subspace
(cf.\ Lemmas~\ref{conncase} and~\ref{insidenil}), the contraction group $(E_f,\alpha_f)$ is simple
as $E_f=\R[X]/f\R[X]$ is a simple $\R[X]$-module
(since~$f$ is irreducible).
If $f,g\in\Omega_\infty$ and $\phi\colon (E_f,\alpha_f)\to(E_g,\alpha_g)$
is an isomorphism of contraction groups, then $\phi$ is linear (cf.\
Lemmas~\ref{conncase} and \ref{insidenil})
and hence an isomorphism of $\R[X]$-modules.
From
\[
\R[X]/f\R[X]= E_f\cong E_g = \R[X]/g\R[X]
\]
we deduce that $f=g$.\vspace{4mm}\,\Punkt

\noindent
{\bf The composition length}
\begin{defn}\label{clength}
Let $(G,\alpha)$ be a locally compact contraction group.
Theorem~\ref{JH}\,(c) shows that each composition series
$\{e\}=G_0\lhd G_1\lhd\cdots\lhd G_n=G$ of contraction groups for
$(G,\alpha)$ has the same length $n\in \N_0$.
We call~$n$ the \emph{composition length} of $(G,\alpha)$
and write $\ell(G):=n$.
\end{defn}
Two simple observations will be useful later.
\begin{la}\label{length-1}
If $(G,\alpha)$ is a locally compact contraction group
and $N\sub G$ an $\alpha$-stable closed normal subgroup, then
\begin{equation}\label{eqlength1}
\ell(G)=\ell(N)+\ell(G/N)
\end{equation}
if we endow $N$ with the contractive automorphism~$\alpha|_N$ and
$G/N$ with\linebreak
$\wb{\alpha}\colon xN\mto\alpha(x)N$.
\end{la}
\begin{proof}
Let $q\colon G\to G/N$ be the canonical quotient
map. Then a composition series for $(N,\alpha|_N)$
and the preimages in~$G$  under~$q$ of a composition series for $(G/N,\wb{\alpha})$
combine to a composition series for $(G,\alpha)$.
\end{proof}
\begin{la}\label{length-2}
If $\phi\colon G\to H$ is a morphism between locally compact
contraction groups $(G,\alpha)$ and $(H,\beta)$, then
\begin{equation}\label{eqlength2}
\ell(G)=\ell(\ker\phi)+\ell(\phi(G)).
\end{equation}
\end{la}
\begin{proof}
Since $\phi(G)\cong G/\ker\phi$, this follows from Lemma~\ref{length-1}.
\end{proof}
$\;$\pagebreak

\noindent
{\bf Composition length in locally {\boldmath pro-$p$} groups}
\begin{la}\label{mod-pro-p}
Let $(G,\alpha)$ be a locally compact contraction group such that
$G$ is locally pro-$p$.
Then the following holds:
\begin{itemize}
\item[{\rm(a)}]
If $G\not=\{e\}$,
then $\Delta(\alpha^{-1})=p^k$ for some $k\in\N$.
\item[{\rm(b)}]
If $G$ is a torsion group, then $\Delta(\alpha^{-1})=p^{\ell(G)}$,
where $\ell(G)$ is the composition length of $(G,\alpha)$.
Moreover, the order of each group element of~$G$ divides $p^{\ell(G)}$.
\end{itemize}
\end{la}
\begin{proof}
(a) Let $V\sub G$ be a compact open subgroup such that $\alpha(V)\sub V$
(see \ref{inv-sub}). Since $\bigcap_{n\in\N_0}\alpha^n(V)=\{e\}$ as a consequence
of~\ref{compacontra}, $\alpha(V)$ is a proper subgroup of~$V$.
Since~$V$ is a pro-$p$-group (see Lemma~\ref{unionprop}) and $\alpha(V)$ is open in~$V$,
we deduce that $\Delta(\alpha^{-1})=[V:\alpha(V)]=p^k$ for some
$k\in\N$ (see (\ref{formul-mod})).

(b) If $H:=\bF_p(\!(t)\!)$ and $\beta(z):=tz$ for $z\in H$, then $U:=\bF_p[\![t]\!]$
is a compact open subgroup of~$H$ and $\beta(U)=t\bF_p[\![t]\!]$
is a subgroup of index~$p$ in~$U$. By~(\ref{formul-mod}),
\begin{equation}\label{bitpiece}
\Delta(\beta^{-1})=[U:\beta(U)]=p.
\end{equation}
If $(G,\alpha)$ has a composition series~(\ref{serial}),
give $G_j/G_{j-1}$ the contractive automorphism~$\alpha_j$
taking $xG_{j-1}$ to $\alpha(x)G_{j-1}$. Since $G_j/G_{j-1}$ is a simple
contraction group, locally pro-$p$, and a torsion group,
the classification in \cite[Theorem~A]{GW} shows that $(G_j/G_{j-1},\alpha_j)\cong (H,\beta)$
with $H=\bF_p(\!(t)\!)$ as above.
Thus
\[
\Delta(\alpha^{-1})=\Delta((\alpha_1)^{-1})\cdots\Delta((\alpha_n)^{-1})
=\Delta(\beta^{-1})^n=p^n=p^{\ell(G)}
\]
indeed (cf.\ \cite[Proposition~III.13.20]{FaD}).
As each subquotient $G_j/G_{j-1}$ is a torsion group of finite exponent~$p$,
we deduce from (\ref{serial}) that~$G$ is a torsion group of exponent
dividing~$p^n=p^{\ell(G)}$.
\end{proof}
\section{Abelian contraction groups}\label{secab}
In this section, we prove the structure theorems
concerning locally compact abelian contraction groups
described in the introduction
(Theorems~D and~E).
In particular, we
classify those locally compact contraction groups
which are abelian torsion groups (Corollary~\ref{thetorso}).
Up to isomorphism,
there are only countably many.
By contrast, for each~$p$ one easily finds uncountable sets of
contraction groups which are abelian $p$-adic Lie groups
and pairwise non-isomorphic,
and likewise for contraction groups which are abelian real Lie groups
(see Example~\ref{uncoup}).
\begin{numba}\label{theabones}
Every locally compact contraction group $(G,\alpha)$
such that $G$ is an abelian $p$-adic Lie group
has an abelian $p$-adic Lie algebra $L(G)$, whence
the BCH-multiplication coincides with addition
and $(G,\alpha)$ is isomorphic to $(L(G),+)$
with the contractive linear automorphism $L(\alpha)$
(as a special case of~\ref{padcase}).
Likewise, every abelian connected locally compact contraction group $(G,\alpha)$
is a real Lie group whose Lie algebra $L(G)$ is abelian,
and $(G,\alpha)$ is isomorphic to $(L(G),+)$
endowed with the contractive linear automorphism~$L(\alpha)$
(cf.~\ref{conncase}).
\end{numba}
\begin{example}\label{uncoup}
For $a\in \Q_p$ of $p$-adic absolute value $|a|_p<1$,
consider the multiplication operator $\mu_a\colon\Q_p\to\Q_p$,
$z\mto az$. Then $(\Q_p,\mu_a)$ is an abelian contraction group.
If also $b\in\Q_p$ with $|b|_p<1$ and $\phi\colon (\Q_p,\mu_a)\to(\Q_p,\mu_b)$
is an isomorphism of contraction groups, then the homomorphism $\phi$ is
$\Q_p$-linear, by Lemma~\ref{insidenil}. Hence
\[
a\phi(1)=\phi(a1)=(\phi\circ\mu_a)(1)=(\mu_b\circ\phi)(1)=b\phi(1),
\]
entailing that $a=b$ and thus $(\Q_p,\mu_a)=(\Q_p,\mu_b)$.
Analogous arguments apply if $\Q_p$ is replaced with~$\R$.
\end{example}
For each $m\in \N$ and prime number~$p$, we write $F_{p^m}:=\Z/p^m\Z$
and endow $F_{p^m}(\!(t)\!)$
with the contractive automorphism $\alpha\colon z\mto tz$.
\begin{rem}\label{andthelengthis}
(a) Note that $(F_{p^m}(\!(t)\!),\alpha)$ has composition length~$m$
since
\[
\{0\}=p^m F_{p^m}(\!(t)\!)\lhd p^{m-1}F_{p^m}(\!(t)\!)\lhd \cdots \lhd p^0 F_{p^m}(\!(t)\!)
=F_{p^m}(\!(t)\!)
\]
is a composition series of $\alpha$-stable closed normal subgroups.\\[2.3mm]
[In fact,
$p^k\big(F_{p^m}(\!(t)\!)\big)$ is the set of all
$x=\sum_{i=i_0}^\infty x_it^i$ in $F_{p^m}(\!(t)\!)$ such that
$i_0\in\Z$ and
$x_i\in p^kF_{p^m}$ for all $i\geq i_0$.
As a consequence,
$p^{k-1}F_{p^m}(\!(t)\!)/p^kF_{p^m}(\!(t)\!)\cong \bF_p(\!(t)\!)$
for each $k\in \{1,\ldots,m\}$.]

(b) Let $x=\sum_{i=i_0}^\infty x_it^i\in F_{p^m}(\!(t)\!)$ with $i_0\in\Z$ and $x_i\in F_{p^m}$
for all $i\geq i_0$.
Let
$k\in\{0,1,\ldots, m\}$.
Then $p^kx=0$ if and only if $p^kx_i=0$ for all $i\geq i_0$,
which holds if and only if $x_i\in p^{m-k}F_{p^m}$. Thus
\begin{equation}\label{thetorp}
\{x\in F_{p^m}(\!(t)\!)\colon p^kx=0\}=p^{m-k}F_{p^m}(\!(t)\!).
\end{equation}
\end{rem}
\begin{la}\label{thegen}
Let $(G,\alpha)$ be a totally disconnected, locally compact contraction group
such that~$G$ is abelian. If
$x\in G$ is an element of prime power order~$p^m$ with $m\in\N$, then there is a unique
morphism of contraction groups
\[
F_{p^m}(\!(t)\!)\to G\;\;\mbox{taking $t^0$ to~$x$}.
\]
Its image $\langle x\rangle_\alpha$ is the smallest
$\alpha$-stable closed subgroup of~$G$ containing~$x$, and its co-restriction
\[
\theta_x\colon F_{p^m}(\!(t)\!)\to\langle x\rangle_\alpha
\]
is an isomorphism of contraction groups if we use $\alpha|_{\langle x\rangle_\alpha}$
on the right hand side. Notably, $\langle x\rangle_\alpha\cong F_{p^m}(\!(t)\!)$.
\end{la}
\begin{proof}
Let~$+$ be the group operation on~$G$.
Let $U\sub G$ be a compact open subgroup such that $\alpha(U)\sub U$ (see \ref{inv-sub}).
Since $\alpha$ is contractive, we have $\alpha^k(x)\in U$ for some $k \in \N_0$.
After replacing $U$ with $\alpha^{-k}(U)$, we may assume that $x\in U$.
Then
\[
U\supseteq \alpha(U)\supseteq \alpha^2(U)\supseteq\cdots
\]
is a basis of identity neighbourhoods of~$G$ consisting of compact open subgroups
(cf.\ \ref{compacontra}).
We claim that a) the limit
\[
\phi_x(z):=\sum_{j=N}^\infty k_j \alpha^j(x)
\]
exists in~$G$ for all $N\in\Z$ and all
\begin{equation}\label{ballstra}
z=\sum_{j=N}^\infty k_jt^j\in F_{p^m}(\!(t)\!)
\end{equation}
with $k_j\in\{0,1,\ldots, p^m-1\}$;
and b) the function $\phi_x\colon F_{p^m}(\!(t)\!)\to G$, $z\mto\phi_x(z)$
is continuous and a homomorphism. As the sets
$V_N$ of all~$z$ as in~(\ref{ballstra}) form an open cover of $F_{p^m}(\!(t)\!)$
for $N\in\Z$ consisting of nested compact open subgroups, it suffices to fix $N\in\Z$ and show
c) that the above
limit exists for $z\in V_N$; and d) that $\phi_x|_{V_N}$ is a continuous homomorphism.
Note that Lemma~\ref{egsum} can be applied
using the continuous functions
\[
f_n\colon V_N\to \alpha^{N+n}(U)\sub G,\quad \sum_{j=N}^\infty
k_jt^j\mto k_{N+n}\alpha^{N+n}(x).\vspace{-1mm}
\]
Hence~c) holds and $\phi_x|_{V_N}$ is continuous. Being a pointwise limit of
group homomorphisms, also~$\phi_x|_{V_N}$ is a homomorphism.
By construction, we have
\[
\phi_x(tz)=\alpha(\phi_x(z))\;\,\mbox{for all $z\in F_{p^m}(\!(t)\!)$.}
\]
Thus $\phi_x$ is a morphism of contraction groups, entailing that its image
$\langle x\rangle_\alpha$ is an $\alpha$-stable closed subgroup of~$G$.
Since the set $F_{p^m}^{\,(\Z)}$ of finitely supported sequences is
dense in $F_{p^m}(\!(t)\!)$, the elements
$\alpha^n(x)$ with $n\in\Z$ generate a dense subgroup of $\langle x\rangle_\alpha$.
As a consequence, $\langle x\rangle_\alpha$ is the smallest $\alpha$-stable
closed subgroup of~$G$ which contains~$x$. To complete the proof, note that
\[
m\leq \ell(\langle x\rangle_\alpha)\leq \ell(F_{p^m}(\!(t)\!))=m
\]
by Lemma~\ref{mod-pro-p}\,(b), \ref{length-2}, and Remark~\ref{andthelengthis}.
Thus $\ell(\langle x\rangle_\alpha)=\ell(F_{p^m}(\!(t)\!))$,
whence $\ell(\ker\theta_x)=0$ (by \ref{length-2}) and thus $\ker\theta_x=\{0\}$.
By \ref{nicemorph}, $\theta_x$ is an isomorphism of topological groups.
\end{proof}
\begin{la}\label{lin-latt}
For each prime $p$ and $m\in\N_0$, the lattice of $\alpha$-stable closed 
subgroups of $F_{p^m}(\!(t)\!)$ is totally ordered; it has the form
\[
\{0\}=p^m F_{p^m}(\!(t)\!)<p^{m-1}F_{p^m}(\!(t)\!)<\cdots< p^0 F_{p^m}(\!(t)\!)=F_{p^m}(\!(t)\!).
\]
If $k\in\{0,1,\ldots, m\}$ and $x\in F_{p^m}(\!(t)\!)$ is an element of order~$p^k$, then
$\langle x\rangle_\alpha=p^{m-k} F_{p^m}(\!(t)\!)$.
\end{la}
\begin{proof}
Let $H$ be an $\alpha$-stable closed subgroup of $F_{p^m}(\!(t)\!)$
and $k\in\{0,\ldots,m\}$ be minimal with $p^kH=\{0\}$.
Let $x\in H$ be an element of order~$p^k$.
Then
\[
\langle x\rangle_\alpha \leq H\leq p^{m-k} F_{p^m}(\!(t)\!),
\]
using Remark~\ref{andthelengthis}\,(b).
As $\langle x\rangle_\alpha\cong F_{p^k}(\!(t)\!)$ (see Lemma~\ref{thegen})
and $p^{m-k}F_{p^m}(\!(t)\!)\cong F_{p^k}(\!(t)\!)$ both have composition length
$k$ (see Remark~\ref{andthelengthis}\,(a)), \ref{length-1} shows that
\[
p^{m-k}F_{p^m}(\!(t)\!)/\langle x\rangle_\alpha=\{0\}
\]
as it has composition length~$0$. Thus $\langle x\rangle_\alpha=H=p^{m-k}F_{p^m}(\!(t)\!)$.\vspace{.7mm}
\end{proof}
{\bf Proof of Theorem~D: Existence.}
By Theorem~A(c),
\[
G=\tor_{p_1}(G)\times\cdots \times \tor_{p_m}(G)
\]
internally as a contraction group
for certain primes $p_1<\cdots<p_m$, such that $\tor_p(G)\not=\{e\}$
for all $p\in\{p_1,\ldots, p_m\}$.
If we can show that $\tor_{p_k}(G)$ is isomorphic to $A_k^{(-\N)}\times A_k^{\N_0}$
with the right shift for a finite abelian group~$A_k$,
then~$G$ is isomorphic to $F^{(-\N)}\times F^{\N_0}$
with the right shift for the finite abelian group $F:=A_1\times\cdots\times A_m$.
After replacing $G$ with $\tor_{p}(G)$ for $p\in\{p_1,\ldots, p_m\}$, we may therefore assume now that
$G$ is a non-trivial $p$-group and locally pro-$p$ for some prime~$p$.
We show by induction on $\ell(G)\in \N$ 
that $G\cong F^{(-\N)}\times F^\N$ for an abelian group~$F$ of order
$|F|=p^{\ell(G)}$. If $\ell(G)=1$, then $(G,\alpha)$ is a simple totally disconnected
contraction group which is a torsion group, abelian and locally pro-$p$,
whence $(G,\alpha)$ is isomorphic to $\bF_p^{(-\N)}\times\bF_p^{\N_0}$
by the classification in \cite[Theorem~A]{GW}, and the assertion holds.
If $\ell(G)>1$, pick an element $x\in G$ of maximal order $p^k$.
Then
\[
\langle x\rangle_\alpha\cong F_{p^k}(\!(t)\!)\cong F_{p^k}^{(-\N)}\times F_{p^k}^{\N_0}
\]
by Lemma~\ref{thegen}, and $k\leq\ell(G)$ by \ref{length-1} and Remark~\ref{andthelengthis}\,(a).
Let $q\colon G\to G/\langle x\rangle_\alpha$ be the canonical quotient morphism
and endow $G/\langle x\rangle_\alpha$ with the contractive automorphism
$\wb{\alpha}$ determined by $\wb{\alpha}(q(y))=q(\alpha(y))$ for $y\in G$.
If $k=\ell(G)$, then $\langle x\rangle_\alpha=G$ is a shift on a restricted power of $F_{p^k}$,
where $|F_{p^k}|=p^k=p^{\ell(G)}$.
If $k<\ell(G)$, then
$G/\langle x\rangle_\alpha$ has composition length $\ell(G)-k$ (see \ref{length-1}),
whence
\[
G/\langle x\rangle_\alpha\cong A^{(-\N)}\times A^{\N_0}
\]
with the right shift for an abelian group~$A$ of order $|A|=p^{\ell(G)-k}$,
by the inductive hypothesis.
The Structure Theorem for Finite Abelian Groups (cf.\ \cite[4.2.6]{Rob} or \cite[Corollary 10.22]{Rot})
shows that
\[
A\cong F_{p^{k_1}}\times \cdots \times F_{p^{k_n}}
\]
for some $n\in\N$ and $k_1,\ldots, k_n\in\N$ such that $k_1+\cdots+k_n=\ell(G)-k$.
Thus
\[
A^{(-\N)}\times A^{\N_0}\cong \bigoplus_{j=1}^n F_{p^{k_j}}(\!(t)\!).\vspace{-1mm}
\]
As a consequence, there exist $y_1,\ldots,y_n\in G/\langle x\rangle_\alpha$
such that
\[
G/\langle x\rangle_\alpha
=\langle y_1\rangle_{\wb{\alpha}}\times\cdots\times\langle y_n\rangle_{\wb{\alpha}}
\]
internally and $F_{p^{k_j}}(\!(t)\!)\cong \langle y_j\rangle_{\wb{\alpha}}$
for all $j\in\{1,\ldots,n\}$,
by means of isomorphisms $\theta_{y_j}$ taking $t^0$ to~$y_j$,
as described in Lemma~\ref{thegen}.\\[2.3mm]
We claim that, for each $j\in \{1,\ldots,n\}$,
there exists $x_j\in G$
such that $q(x_j)=y_j$ and $x_j$ has order $p^{k_j}$.
If this is true, then
\[
\sigma_j:=\theta_{x_j}\circ \theta_{y_j}^{-1}\colon \langle y_j\rangle_{\wb{\alpha}}\to
\langle x_j\rangle_\alpha
\]
is a morphism of contraction groups for $j\in\{1,\ldots,n\}$,
whence also
\[
\sigma\colon G/\langle x\rangle_\alpha\to G,\quad g_1+\cdots +g_n\mto \sigma_1(g_1)+\cdots+\sigma_n(g_n)
\]
for $(g_1,\ldots, g_n)\in \langle y_1\rangle_{\wb{\alpha}}\times\cdots\times\langle y_n
\rangle_{\wb{\alpha}}$
is a morphism of contraction groups. Then $q\circ\sigma$ is a morphism
of contraction groups which takes $y_j$ to itself and hence
coincides with the identity map on $\langle y_j\rangle_{\wb{\alpha}}$,
for each $j\in\{1,\ldots,n\}$. Thus
\[
q\circ \sigma=\id
\]
on $G/\langle x\rangle_\alpha$, entailing that
$G=\im(\sigma)\times \ker(q)=\langle x_1\rangle_\alpha\times\cdots\times\langle x_n\rangle_\alpha
\times \langle x\rangle_\alpha$
internally as a topological group (and a contraction group).
Thus $G\cong F^{(-\N)}\times F^{\N_0}$ with
$F:=A\times F_{p_k}$.\\[2.3mm]
To prove the claim, pick $z_j\in G$ such that $q(z_j)=y_j$.
Then $z_j$ has order $p^{\ell_j}$ with $k_j\leq\ell_j\leq k$, and
$p^{k_j}z_j\in \langle x\rangle_\alpha$ is an element of order $p^{\ell_j-k_j}$
and hence contained in
\[
p^{k-\ell_j+k_j}\langle x\rangle_\alpha,
\]
by Lemma~\ref{lin-latt}.
Thus $p^{k_j}z_j=p^{k-\ell_j+k_j}r_j$
for some $r_j\in \langle x\rangle_\alpha$
and thus $p^{k_j}z_j=p^{k_j}(p^{k-\ell_j}r_j)$
with $k-\ell_j\geq 0$.
Then $x_j:=z_j-p^{k-\ell_j}r_j$ is an element such that $q(x_j)=q(z_j)=y_j$
(whence its order is at least $p^{k_j}$)
and $p^{k_j}x_j=0$, whence $x_j$ has order~$p^{k_j}$.\\[2.3mm]
The uniqueness assertion will be proved at the end of this section.\vspace{2mm}\,\Punkt

\noindent
{\bf Proof of Theorem~E.}
Existence:
By Theorem~A, there are prime numbers $p_1<\cdots < p_n$
and
$p$-adic Lie groups $G_p$ for $p\in\{p_1,\ldots,p_n\}$ which are $\alpha$-stable closed
subgroups of~$G$ such that
\[
G=G_e\times G_{p_1}\times\cdots\times G_{p_n}\times \tor(G);
\]
moreover, $\tor(G)\cong F^{(-\N)}\times F^{\N_0}$ for a finite abelian group~$F$,
by Theorem~D. Abbreviate $G_\infty:=G_e$. If $\beta$ is an endomorphism
of a finite-dimensional vector space~$E$ over a field~$\K$,
let us write $(E,\beta)$ for~$E$, endowed with the $\K[X]$-module
structure determined by $X.v=\beta(v)$ for $v\in E$.
Let $p\in\{p_1,\ldots,p_n,\infty\}$.
In view of \ref{theabones}, the Structure Theorem for Finitely Generated Modules
over Principal Ideal Domains shows that
\begin{equation}\label{preformul}
(G_p,\alpha|_{G_p})\cong \bigoplus_f\bigoplus_{n\in\N}
(E_{f^n},\alpha_{f^n})^{\mu(p,f,n)},\vspace{-1mm}
\end{equation}
where $f$ ranges through the set of all monic irreducible
polynomials with coefficients in $\Q_p$ and $\mu(p,f,n)\not=0$
for only finitely many $(f,n)$. As each $\alpha_{f^n}$ with $\mu(p,f,n)>0$ is contractive
(like $\alpha|_{G_p}$), we must have $f\in\Omega_p$ for such $(p,f,n)$
(see Lemmas~\ref{reallin} and \ref{complxlin}).
We may therefore assume that $f\in\Omega_p$ in~(\ref{preformul}).
For each prime $p$ such that $p\not\in\{p_1,\ldots, p_n\}$, we set $\mu(p,f,n):=0$
for all $f\in\Omega_p$ and $n\in\N$.
Finally, by the Structure Theorem for Finite Abelian Groups
(cf.\ \cite[4.2.6]{Rob} or \cite[Corollary 10.22]{Rot}),
\[
F\cong \bigoplus_{p\in\bP}\bigoplus_{n\in\N}(F_{p^n})^{\nu(p,n)},\vspace{-1mm}
\]
where $\nu(p,n)\in\N_0$ is non-zero for only finitely many $(p,n)$.
Hence
\[
F^{({-\N})}\times F^{\N_0}\cong \bigoplus_{p\in\bP}\bigoplus_{n\in\N}
(F_{p^n}(\!(t)\!))^{\nu(p,n)},\vspace{-1mm}
\]
showing that $(G,\alpha)$ is isomorphic to
a direct sum of the form~($*$) asserted in~Theorem~E.
Conversely, every group of the form~($*$)
described in Theorem~E is a locally compact contraction group
(using Lemmas~\ref{reallin}, \ref{complxlin},
and \cite[Lemma~3.4 and its proof]{Wan}).
\\[2.3mm]
Uniqueness: If ($*$) holds, then $(G_e,\alpha|_{G_e})$ is isomorphic to the connected component
$\bigoplus_{f\in\Omega_\infty}\bigoplus_{n\in\N}(E_{f^n},\alpha_{f^n})^{\mu(\infty,f,n)}$
and uniquely determines the integers $\mu(\infty,f,n)$ by uniqueness
of the $\mu(\infty,f,n)$ in~(\ref{preformul}).
For each prime $p$ such that $p\not\in\{p_1,\ldots, p_n\}$, we must have $\mu(p,f,n)=0$
for all $f\in\Omega_p$ and $n\in\N$, by uniqueness in Theorem~A(b).
For $p\in\{p_1,\ldots,p_n\}$,
uniqueness in Theorem~A(b) implies that the contraction group $(G_p,\alpha|_{G_p})$ is isomorphic
to $\bigoplus_{f\in\Omega_p}\bigoplus_{n\in\N}(E_{f^n},\alpha_{f^n})^{\mu(p,f,n)}$
and uniquely determines the $\mu(p,f,n)$ by uniqueness
of the $\mu(p,f,n)$ in~(\ref{preformul}).
Finally, if ($*$) holds then
\begin{equation}\label{henceget}
(\tor_p(G),\alpha|_{\tor_p(G)})\cong \bigoplus_{n\in\N}
(F_{p^n}(\!(t)\!),\alpha_{p^n})^{\nu(p,n)}\vspace{-1mm}
\end{equation}
for each prime number~$p$. We may therefore assume that $G=\tor_p(G)$ to establish
uniqueness of the $\nu(p,n)$ for $n\in\N$ and a given prime number~$p$.
For $m\in\N$, let~$T_m$ be the $\alpha$-stable closed subgroup of~$G$ consisting
of elements $x\in G$ whose order divides~$p^m$.
Then~$T_m/T_{m-1}$ is trivial if~$m$ is greater than the largest~$n$ such that $\nu(p,n)\ne0$ and is isomorphic to $(F_{p}(\!(t)\!),\alpha_{p})^{\nu(p,n)}$ if~$m$ equals this value. This isomorphism determines~$\nu(p,n)$ uniquely in this case. 
For smaller values of~$m$, we have
\begin{equation}
S_{m}/S_{m-1}\cong \bigoplus_{n\geq m}
(F_{p}(\!(t)\!),\alpha_{p})^{\nu(p,n)},
\end{equation}
and~$\nu(p,m)$ may thus be determined once~$\nu(p,n)$ is known for all~$n>m$.
\,\vspace{2.3mm}\Punkt

\noindent
The following is immediate from Theorem~E and the preceding discussion.
\begin{cor}\label{thetorso}
Endowed with the right shift $\alpha_{p^n}$ in each component,
\[
\bigoplus_{p\in\bP}\bigoplus_{n\in\N}(\bF_{p^n}(\!(t)\!),\alpha_{p^n})^{\nu(p,n)}
\]
is a locally compact abelian torsion contraction group
for each family
\[
(\nu(p,n))_{(p,n)\in\bP\times\N}\in\N_0^{(\bP\times\N)}.
\]
Any locally compact abelian torsion contraction group
is isomorphic to a contraction group of the preceding form, for
a uniquely determined family $(\nu(p,n))_{(p,n)\in\bP\times\N}$.
Notably, there are only countably many locally compact abelian
contraction groups which are torsion groups, up to isomorphism.\,\vspace{1mm}\Punkt
\end{cor}

\noindent
{\bf Proof of Theorem~D, completed: Uniqueness.}
Assume that $(G,\alpha)$ is isomorphic to $F^{(-\N)}\times F^{\N_0}$
with the right shift for a finite abelian group~$F$.
By the Structure Theorem for Finite Abelian Groups,
\[
F\, \cong\, \bigoplus_{p\in\bP}\bigoplus_{n\in\N}(F_{p^n})^{\nu(p,n)}\vspace{-1mm}
\]
for $\nu(p,n)\in\N_0$ which are non-zero for only finitely many $(p,n)$.
Then
\begin{equation}\label{wildeter}
(G,\alpha)\, \cong \, \bigoplus_{p\in\bP}\bigoplus_{n\in\N}(F_{p^n}(\!(t)\!))^{\nu(p,n)}\vspace{-1mm}
\end{equation}
as a contraction group. By Theorem~E, the $\nu(p,n)$ in~(\ref{wildeter})
are uniquely determined by $(G,\alpha)$. Hence
$F$ is determined up to isomorphism.\,\Punkt
\section{\!\!Two auxiliary results concerning extensions}\label{secauxi}
We record two observations concerning extensions, for later use.
\begin{la}\label{possible-c}
Let $(A,\alpha_A)$, $(\widehat{G},\wh{\alpha})$ and $(G,\alpha)$
be totally disconnected, locally compact contraction groups. Let
\[
\{e\}\to A\stackrel{\iota}{\to} \wh{G}\stackrel{q}{\to} G \to \{e\}
\]
be an extension of contraction groups such that $(A,\alpha_A)$
and $(G,\alpha)$ are simple contraction groups
and both~$A$ and~$G$ are abelian. If~$\wh{G}$ has non-trivial centre,
then
$\wh{G}$ is abelian
or $Z(\wh{G})=\iota(A)$ and~$G$ is $2$-step nilpotent.
\end{la}
\begin{proof}
After replacing~$A$ with $\iota(A)$,
we may assume that $A\sub \wh{G}$ and~$\iota$ is the inclusion map.
Since $Z(\wh{G})$ is a closed $\wh{\alpha}$-stable normal subgroup of~$\wh{G}$
and $q\colon \wh{G}\to G$ is a surjective morphism of contraction groups, the image
$q(Z(\wh{G}))$ is a closed $\alpha$-stable normal subgroup of~$G$
(see~\ref{nicemorph}).
Hence $q(Z(\wh{G}))=\{e\}$ or $q(Z(\wh{G}))=G$.
In the first case, $Z(\wh{G})$ is a non-trivial, $\alpha_A$-stable closed normal subgroup of~$A$,
whence $Z(\wh{G})=A$ by simplicity of $(A,\alpha_A)$; as a consequence, $\wh{G}$ is 2-step nilpotent.
In the second case, $\wh{G}=Z(\wh{G})A$ is abelian.
\end{proof}
In the next lemma, we use notation as in Appendix~\ref{cohom}.
Note that,
in the situation of~\ref{thesetapp},
the trivial homomorphism $\gamma\colon G\to \Aut(A)$, $x\mto \id_A$
always satisfies~(\ref{covar}).
\begin{la}\label{lessknown}
Let $(G,\alpha)$ and $(A,\alpha_A)$
be totally disconnected, locally compact contraction groups with~$A$ abelian.
Let
$\omega_1,\omega_2\in Z^2_{\eq}(G,A)$
with respect to the trivial homomorphism
$\gamma\colon G\to\Aut(A)$, $x\mto \id_A$ and
\[
\phi\colon A\times_{\omega_1} G\to A\times_{\omega_2}G
\]
be an isomorphism of contraction groups such that $\phi(A\times\{e\})=A\times\{e\}$.
Let $\phi_A$ be the automorphism of the contraction group~$A$ determined by $\phi(x,e)=(\phi_A(x),e)$
for $x\in A$ and $\wb{\phi}$ be the automorphism of the contraction group~$G$
determined by $\wb{\phi}(\pr_2(x,y))=\pr_2(\phi(x,y))$ for $x\in A$ and $y\in G$.
Then there exists $\beta\in B^2_{\eq}(G,A)$ such that
\[
\omega_2=\phi_A\circ\omega_1\circ(\wb{\phi}^{-1}\times\wb{\phi}^{-1})+\beta,
\]
using additive notation for~$A$.
\end{la}
\begin{proof}
It is immediate that
\[
\omega_0:=\phi_A\circ \omega_1\circ (\wb{\phi}^{-1}\times\wb{\phi}^{-1})\in Z^2_{\eq}(G,A)
\]
and
a straightforward calculation shows that the equivariant homeomorphism
\[
\phi_A^{-1}\times\wb{\phi}^{-1}\colon
 A\times_{\omega_0} G\to A\times_{\omega_1}G,\;\, (x,y)\mto(\phi_A^{-1}(x),\wb{\phi}^{-1}(y))
\]
is a group homomorphism (and thus an isomorphism of contraction groups).
Hence also
\[
\psi:=\phi\circ
(\phi_A^{-1}\times\wb{\phi}^{-1})\colon
A\times_{\omega_0} G\to A\times_{\omega_2}G
\]
is an isomorphism of contraction groups.
By construction, we have
\[
\psi(x,e)=\phi(\phi_A^{-1}(x),e)=(\phi_A(\phi_A^{-1}(x)),e)=(x,e)
\]
for all $x\in A$ and thus $\psi\circ\iota=\iota$
with $\iota\colon A\to A\times G$, $x\mto (x,e)$.
Moreover, $\pr_2\circ\, \psi=\pr_2$ since $\pr_2\circ \, \phi=\wb{\phi}\circ \pr_2$.
Hence $\psi$ is an isomorphism between the extensions
\[
\{e\}\to A\stackrel{\iota}{\to} A\times_{\omega_0} G\stackrel{\pr_2}{\to}\{e\}
\;\;\mbox{and}\;\;
\{e\}\to A\stackrel{\iota}{\to} A\times_{\omega_2} G\stackrel{\pr_2}{\to}\{e\}
\]
of totally disconnected, locally compact contraction groups.
Therefore\linebreak
$\beta:=\omega_2-\omega_0\in B^2_{\eq}(G,A)$
(see Proposition~\ref{soinj}).
\end{proof}
\section{Description of biadditive cocycles}\label{secbiadd}
Let $p$ be a prime, $A:=\bF_p(\!(t)\!)$ and $\alpha\colon A\to A$ be the contractive
automorphism $f\mto tf$. We let $A$ act on itself via
$\gamma\colon A\to\Aut(A)$, $x\mto\id_A$,
i.e., we consider central extensions of~$A$ by~$A$.
Thus, a continuous equivariant 2-cocycle $\omega\colon A\times A\to A$
is a continuous map satisfying the cocycle identity
\begin{equation}\label{concreteid}
\omega(y,z)-\omega(x+y,z)+\omega(x,y+z)-\omega(x,y)=0\;\,\mbox{for all $x,y,z\in A$}
\end{equation}
and the equivariance condition $\alpha\circ \omega=\omega\circ(\alpha\times\alpha)$,
i.e.,
\[
t\omega(x,y)=\omega(tx,ty)\;\,\mbox{for all $x,y\in A$.}
\]
As a consequence,
\begin{equation}\label{tnequi}
t^k\omega(x,y)=\omega(t^kx,t^ky)\;\,\mbox{for all $x,y\in A$ and $k\in\Z$.}
\end{equation}
The cocycle condition (\ref{concreteid}) is satisfied
in particular if a map $\omega\colon A\times A\to A$ is \emph{biadditive},
i.e., $\omega(x+y,z)=\omega(x,z)+\omega(y,z)$ for all $x,y,z\in A$ and
$\omega(x,y+z)=\omega(x,y)+\omega(x,z)$.
Then $\omega(x,0)=\omega(0,x)=0$ for all $x\in A$.\\[2mm]
If $\omega\colon A\times A\to A$ is a continuous, equivariant, biadditive map,
we define
\begin{equation}\label{defnbm}
b_m(\omega):=\omega(t^0,t^m)\;\,\mbox{for $m\in\Z$.}
\end{equation}
Then
\[
b_m(\omega)=\omega(t^0,t^m)\to\omega(t^0,0)=0
\]
as $m\to\infty$ and
\[
t^mb_{-m}(\omega)=t^m\omega(t^0,t^{-m})=\omega(t^m,t^0)\to 0
\]
as $m\to\infty$. The two-sided sequences $(b_m(\omega))_{m\in\Z}$
are therefore elements of the set~$B$
of all $(a_m)_{m\in \Z}$ with $a_m\in A$ such that
\begin{equation}\label{behav}
\mbox{$a_m\to 0$ and $t^ma_{-m}\to 0$ as $m\to\infty$.}
\end{equation}
Let $Z^2_{\bil}(A,A)\sub Z^2_{\eq}(A,A)$ be the set of all continuous, equivariant,
biadditive (and hence $\bF_p$-bilinear) maps $\omega\colon A\times A\to A$.
For example, for each $n\in\Z$ we have
$\omega_n\in Z^2_{\bil}(A,A)$ if we define
\begin{equation}\label{omegam-1}
\omega_n(x,y):=\sum_{i=i_0}^\infty x_iy_{i+n}t^i
\end{equation}
for all $x=\sum_{i=i_0}^\infty x_it^i$, $y=\sum_{j=j_0}^\infty y_jt^j$
with $i_0,j_0\in\Z$ and $x_i,y_j\in\bF_p$ for integers $i\geq i_0$
and $j\geq j_0$ (and $y_j:=0$ for integers $j<j_0$ here
and in the following).
Note that the limit in (\ref{omegam-1})
exists and defines a continuous $A$-valued function
on each ball
\[
\wb{B}^A_{p^{-i_0}}(0):=\{x\in A\colon |x|\leq p^{-i_0}\}
\]
(and hence on~$A$) as
\[
f_i\colon \wb{B}^A_{p^{-i_0}}(0)\times A\to t^i \bF_p[\![t]\!],\quad (x,y)\mto x_iy_{i+n}t^i
\]
is a continuous mapping for each integer $i\geq i_0$, and Lemma~\ref{egsum} applies.\\[2.3mm]
We mention that
\begin{equation}\label{projectcomp}
\omega_n(t^{i_0},t^{j_0})=\sum_{i=i_0}^\infty \delta_{i_0,i}\delta_{j_0,i+n}t^i
= \delta_{j_0,i_0+n}t^{i_0} \;\; \mbox{for all $i_0,j_0\in\Z$}
\end{equation}
(using Kronecker's delta, considered as an element of $\bF_p\sub\bF_p(\!(t)\!)$).
Notably,
\begin{equation}\label{procomp2}
\omega_n(t^0,t^j)=\delta_{n,j}\;\mbox{for all $j,n\in\Z$.}
\end{equation}
Observe that
\begin{equation}\label{very-basic}
|\omega_n(x,y)|\leq |x|\;\mbox{for all $n\in\Z$ and $x,y\in A$.}
\end{equation}
This is clear if $x=0$, as $\omega_n(x,y)=0$ then.
If $x\not=0$, write $x=\sum_{i=i_0}^\infty x_it^i$ with $x_{i_0}\not=0$ and $y=\sum_{j=j_0}y_jt^j$.
By (\ref{omegam-1}), $|\omega_n(x,y)|\leq p^{-i_0}=|x|$.
\begin{prop}\label{paracoc}
The mapping $b\colon Z^2_{\bil}(A,A)\to B$, $\omega\mto (b_m(\omega))_{m\in\Z}$
is a bijection.
\end{prop}
\begin{proof}
The domain and range of $b$ are groups under pointwise addition and $b$ is a homomorphism.
If $b(\omega)=0$, then $0=b_m(\omega)=\omega(t^0,t^m)$ for each $m\in\Z$
and hence also $\omega(t^n,t^m)=t^n\omega(t^0,t^{m-n})=t^nb_{m-n}(\omega)=0$.
Since $\omega$ is continuous biadditive and $\{t^n\colon n\in\Z\}$ generates a dense
subgroup of~$A$, we deduce that $\omega=0$. Thus $\ker(b)=\{0\}$ and
hence~$b$ is injective.
To see that~$b$ is surjective, let $a=(a_m)_{m\in\Z}\in B$.
We claim that both series needed to define
\begin{equation}\label{def-om-a}
\omega_a(x,y):=\sum_{n=0}^\infty a_n\omega_n(x,y)+\sum_{n=1}^\infty a_{-n}\omega_{-n}(x,y)
\end{equation}
are convergent and define continuous $A$-valued functions of $(x,y)\!\in\! A\!\times \!A$.
Then $\omega_a$ is biadditive
and equivariant, as these properties are inherited by $\bF_p(\!(t)\!)$-linear combinations
of biadditive equivariant maps
and pointwise limits of such. Hence $\omega_a\in Z^2_{\bil}(A,A)$.
As $\omega_n(t^0,t^m)=\delta_{n,m}$ by (\ref{procomp2}), we have
\[
\omega_a(t^0,t^m)
=\sum_{n=0}^\infty a_n\omega_n(t^0,t^m)
+\sum_{n=1}^\infty a_{-n}\omega_{-n}(t^0,t^m)=a_m
\]
for all $m\in \Z$ and thus $b(\omega_a)=a$, whence~$b$ is surjective and hence bijective.
To prove the claim, let $i_0\in\Z$.
The sequence $(\rho_n)_{n\in\N_0}$ with
\[
\rho_n:=\sup\{|a_k|p^{-i_0}\colon k\geq n\}
\]
is monotonically decreasing and $\lim_{n\to\infty}\rho_n=0$.
Let $x:=\sum_{i=i_0}^\infty x_it^i$ and $y:=\sum_{i=i_0}^\infty y_it^i$
with $x_i,y_i\in\bF_p$ for $i\geq i_0$. 
For each $k\in \N_0$, we have
\[
|a_k\omega_k(x,y)|=|a_k|\,\left|\sum_{i=i_0}^\infty x_iy_{i+k}t^i\right|
\leq |a_k|\,|t^{i_0}|=|a_k|p^{-i_0}\leq \rho_k
\]
and thus $a_k\omega_k(x,y)\in\wb{B}^A_{\rho_k}(0)$.
Using Lemma~\ref{egsum},
we deduce that the first sum in (\ref{def-om-a}) converges
to a continuous function of $(x,y)\in \wb{B}^A_{p^{-i_0}}(0)\times
\wb{B}^A_{p^{-i_0}}(0)$.

Also $\tau_n:=\sup\{|t^ka_{-k}|p^{-i_0}\colon k\geq n\}$
defines a monotonically decreasing sequence $(\tau_n)_{n\in\N}$
such that $\lim_{n\to\infty}\tau_n=0$.
For $k\in\N$ we have
\begin{eqnarray*}
|a_{-k}\omega_{-k}(x,y)| &=& |a_{-k}|\,\left|\sum_{i=i_0+k}^\infty x_iy_{i-k}t^i\right|
=|a_{-k}|\,|t^k|\left|\sum_{j=i_0}^\infty x_{j+k}y_j t^j\right|\\
& \leq & |a_{-k}t^k|\,|t^{i_0}|=|a_{-k}t^k| p^{-i_0}\leq \tau_k.
\end{eqnarray*}
Using Lemma~\ref{egsum}, we find that
$\sum_{k=1}^\infty a_{-k}\omega_{-k}(x,y)$ converges
for all $(x,y)\in \wb{B}^A_{p^{-i_0}}(0)\times
\wb{B}^A_{p^{-i_0}}(0)$
and defines a continuous $A$-valued function thereof.
\end{proof}
\section{An uncountable set of contraction groups}\label{secex}
For $A$ as in Section~\ref{secbiadd} and $s=(s_n)_{n\in\N}\in \{0,1\}^\N$,
we define a function $\eta_s\colon A\times A\to A$ via
\[
\eta_s(x,y):=\sum_{n=1}^\infty s_n t^n \omega_{2n}(x,y)\;\,\mbox{for $x,y\in A$,}
\]
with $\omega_n$ as in~(\ref{omegam-1}).
Thus $\eta_s=\omega_a$ (as in (\ref{def-om-a})) with $a_{2n}:=s_nt^n$ for all $n\in\N$,
$a_n:=0$ for $n\in -\N_0$ or $n\in 2\N-1$. As a consequence,
$\eta_s\in Z^2_{\bil}(A,A)$. Note that $\{0,1\}^\N$ is an uncountable set of
continuum cardinality. We show:
\begin{thm}\label{mainres}
The contraction groups $(A\times_{\eta_s} A,\alpha\times \alpha)$
for $s\in \{0,1\}^\N$ are pairwise not isomorphic.
Thus, if $r,s\in\{0,1\}^\N$ and there exists an isomorphism
$\psi\colon A\times_{\eta_s} A\to A\times_{\eta_r} A$ of topological groups
such that $(\alpha\times\alpha)\circ\psi=\psi\circ (\alpha\times\alpha)$,
then $r=s$.
In particular, the map
\[
\{0,1\}^\N\to H^2_{\eq}(A,A),\quad s\mto [\eta_s]
\]
is injective.
\end{thm}
Several lemmas will help us to establish the theorem.
Given $a\in \bF_p(\!(t)\!)^\times$, let $\mu_a\colon \bF_p(\!(t)\!)\to\bF_p(\!(t)\!)$
be the multiplication operator $x\mto xa$.
Since $\mu_a(tx)=txa=t\mu_a(x)$ for all $x\in A$,
we see that $\mu_a$ is an element of the group $\Aut(A,\alpha)$
of all automorphisms of the contraction group $(A,\alpha)$.
\begin{la}\label{eqAuto}
The map $\ve\colon \Aut(A,\alpha)\to \bF_p(\!(t)\!)^\times$, $\phi\mto \phi(1)$
is an isomorphism of groups with inverse $a\mto\mu_a$.
\end{la}
\begin{proof}
Since $\mu_a(1)=a$ for each $a\in \bF_p(\!(t)\!)$,
the map~$\ve$ is surjective. If $\phi \in\Aut(A,\alpha)$ and $a:=\phi(1)$,
then $\phi(t^i)=(\phi\circ\alpha^i)(1)=(\alpha^i\circ\phi)(1)=t^ia=\mu_a(t^i)$
for each $i\in\Z$. As a consequence, for each $x\in A$, say $x=\sum_{i=i_0}^\infty x_it^i$
with $i_0\in\Z$ and $x_i\in\bF_p$, we have
\[
\phi(x)=\sum_{i=i_0}^\infty x_i\phi(t^i)=\sum_{i=i_0}^\infty x_it^ia=xa,
\]
whence $\phi=\mu_a$. Thus $\phi$ is determined by $a=\ve(\phi)$ and thus $\ve$
is injective. Since $\ve^{-1}\colon a\mto \mu_a$ is a group homomorphism,
$\ve$ is an isomorphism.
\end{proof}
\begin{la}\label{main-estimate}
Let $s\in \{0,1\}^\N$ and $x,y\in A$. Let $n_0\in\N$ such that
$s_n=0$ for all $n\in \{1,2,\ldots, n_0-1\}$.
Then
\begin{equation}\label{two-versions}
|\eta_s(x,y)| \leq p^{-n_0}|x|.
\end{equation}
Notably, $|\eta_s(x,y)|\leq p^{-1}|x|<|x|$ if $x\not=0$.
\end{la}
\begin{proof}
For all $x,y\in A$, we have
\begin{equation}\label{partsum}
\eta_s(x,y)=\sum_{n=n_0}^\infty s_nt^n\omega_{2n}(x,y),
\end{equation}
where $|s_nt^n\omega_{2n}(x,y)|\leq p^{-n}|\omega_{2n}(x,y)|\leq p^{-n}|x|\leq p^{-n_0}|x|$
for all $n\geq n_0$, using~(\ref{very-basic}) for the second inequality.
We deduce that each finite partial sum in~(\ref{partsum})
has absolute value $\leq p^{-n_0}|x|$, whence so does the limit $\eta_s(x,y)$.
\end{proof}
\begin{la}\label{centreA}
For each $s\in \{0,1\}^\N$ with $s\not=0$,
the centre of $A\times_{\eta_s} A$ is $A\times\{0\}$.
If $s=0$, then $A\times_{\eta_s}A=A\times A$ is abelian.
\end{la}
\begin{proof}
The last assertion being trivial, let us assume that $s\not=0$ now.
Let $x=(a_1,b_1)\in A\times_{\eta_s}A$.
Then $xy=(a_1+a_2+\eta_s(b_1,b_2),b_1+b_2)$
and $yx=(a_1+a_2+\eta_s(b_2,b_1),b_1+b_2)$ for all $(a_2,b_2)\in A\times_{\eta_s} A$,
whence $x$ is in the centre of $A\times_{\eta_s} A$ if and only if
\begin{equation}\label{charzenta}
\eta_s(b_1,b_2)=\eta_s(b_2,b_1)\;\;\mbox{for all $b_2\in A$.}
\end{equation}
This condition is satisfied if $b_1=0$. We show that, conversely,
(\ref{charzenta}) implies $b_1=0$, from which the assertion follows.
In fact, if $b_1\not=0$, then $b_1=\sum_{i=i_0}^\infty x_it^i$
for some $i_0\in\Z$ and elements $x_i\in \bF_p$ with $x_{i_0}\not=0$.
Let $n_0\in\N$ be minimal subject to $s_{n_0}\not=0$.
Then
\[
\eta_s(b_1,t^{2n_0+i_0})=s_{n_0} t^{n_0} x_{i_0} t^{i_0}
\]
with $|\eta_s(b_1,t^{2n_0+i_0})|=p^{-n_0-i_0}$
but
$|\eta_s(t^{2n_0+i_0},b_1)|< |t^{2n_0+i_0}|=p^{-2n_0-i_0}<p^{-n_0-i_0}$,
by Lemma~\ref{main-estimate}.
Thus $\eta_s(t^{2n_0+i_0},b_1)\not=\eta_s(b_1,t^{2n_0+i_0})$.
\end{proof}
{\bf Proof of Theorem~\ref{mainres}.}
Let $\psi\colon A\times_{\eta_s} A\to A\times_{\eta_r} A$
be as described in the theorem. Lemma~\ref{centreA}
implies that $r=0$ if and only if $s=0$. We may therefore assume now that
both~$r$ and~$s$ are non-zero.\\[2.3mm]
The isomorphism $\psi$ maps the centre of
$A\times_{\eta_s} A$ onto the centre of $A\times_{\eta_r} A$.
Using Lemma~\ref{centreA}, we deduce that
\begin{equation}\label{prerequ}
\psi(A\times \{0\})=A\times \{0\}.
\end{equation}
In view of (\ref{prerequ}), we can apply Lemma~\ref{lessknown},
and deduce that there are $\sigma,\tau\in\Aut(A,\alpha)$
and $\beta\in B^2_{\eq}(A,A)$ such that
\[
\eta_r=\sigma\circ\eta_s\circ (\tau\times\tau)+\beta.
\]
By Lemma~\ref{eqAuto},
there exist $a,b\in\bF_p(\!(t)\!)^\times$ such that
$\sigma(x)=ax$ and $\tau(x)=bx$ for all $x\in A$.
Thus
\[
\eta_r(x,y)=a\eta_s(bx,by)+\beta(x,y)\;\,\mbox{for all $x,y\in A$.}
\]
Let $k\in\Z$ such that $|t^kb|=1$.
Since $\eta_s(bx,by)=t^{-k}\eta_s(t^kbx,t^kby)$ for all $x,y\in A$
by equivariance, after replacing $b$ with $t^kb$ and $a$ with $t^{-k}a$ we
may assume that $|b|=1$.
There exists a continuous function $f\colon A\to A$ with $f(0)=0$ such that
\[
\beta(x,y)=f(x)+f(y)-f(x+y)\;\,\mbox{for all $x,y\in A$.}
\]
Then $\beta$ is symmetric:
\[
\beta(x,y)=\beta(y,x)\quad\mbox{for all $x,y\in A$.}
\]
Hence, we have
\begin{equation}\label{ingr2}
\eta_r(x,y)-\eta_r(y,x)=a\eta_s(bx,by)-a\eta_s(by,bx)\;\,\mbox{for all $x,y\in A$.}
\end{equation}
Let $n_0\in\N$ be minimal such that $r_{n_0}\not=0$;
let $m_0\in\N$ be minimal with $s_{m_0}\not=0$.
We claim that
\begin{equation}\label{size}
|\eta_r(x,y)-\eta_r(y,x)|=p^{-n_0}
\end{equation}
for all $x,y\in A$ such that $|x|=1$ and $|y|=p^{-2n_0}$.
If this is true, then the left-hand side of the identity
\begin{equation}\label{ingr3}
\eta_r(b^{-1}t^0,b^{-1}t^{2n_0})\!-\!\eta_r(b^{-1}t^{2n_0},b^{-1}t^0)=
a\eta_s(t^0,t^{2n_0})\!-\!\underbrace{a\eta_s(t^{2n_0},t^0)}_{=0}=at^{n_0}s_{n_0}
\end{equation}
(which follows from (\ref{ingr2}) and (\ref{procomp2})) has absolute value
$p^{-n_0}$, whence it is non-zero.
Hence also the right-hand side is non-zero and thus $s_{n_0}\not=0$,
whence $m_0\leq n_0$. Hence $s_{n_0}=1$,
whence the right-hand side of~(\ref{ingr3}) has absolute value $|a|p^{-n_0}$.
As the left-hand side has absolute value~$p^{-n_0}$,
we deduce that
\begin{equation}\label{moda}
|a|=1.
\end{equation}
Reversing the roles of~$r$ and~$s$,
we see that also $n_0\leq m_0$ and thus $n_0=m_0$.\\[2.3mm]
To establish the claim (\ref{size}), write $x=\sum_{i=0}^\infty x_it^i$
and $y=\sum_{j=2n_0}^\infty y_jt^j$
with $x_i,y_j \in \bF_p$ and $x_0\not=0$ as well as $y_{2n_0}\not=0$.
Then
\begin{equation}\label{ingr4}
\eta_r(x,y)=\eta_r(x_0t^0,y_{2n_0}t^{2n_0})+\eta_r(x_0t^0,y')+\eta_r(x',y)
\end{equation}
with $x':=\sum_{i=1}^\infty x_it^i$ and $y':=\sum_{j=2n_0+1}^\infty y_jt^j$.
The first summand in (\ref{ingr4}) is $t^{n_0}r_{n_0} x_0y_{2n_0}=t^{n_0}x_0y_{2n_0}$,
which has absolute value $p^{-n_0}$.
Since $\omega_{2n_0}(x_0t^0,y')=0$, the second summand in (\ref{ingr4}) is
\[
\sum_{n=n_0+1}^\infty r_nt^n\omega_{2n}(x_0t^0,y');
\]
by Lemma~\ref{main-estimate}, its absolute value is
$\leq p^{-n_0-1}|x_0t^0|=p^{-n_0-1}<p^{-n_0}$.
As $|x'|\leq p^{-1}$, Lemma~\ref{main-estimate} yields
$|\eta_r(x',y)|\leq p^{-n_0}|x'|\leq p^{-n_0-1}<p^{-n_0}$
for the third summand in (\ref{ingr4}).
Using~(\ref{ultra}),
we deduce that
\begin{equation}\label{ingr5}
|\eta_r(x,y)|=p^{-n_0}.
\end{equation}
Now $|\eta_r(y,x)|\leq p^{-n_0}|y|=p^{-n_0}p^{-2n_0}<p^{-n_0}$,
by Lemma~\ref{main-estimate}.
Thus, using~(\ref{ultra}), we find that
$|\eta_r(x,y)-\eta_r(y,x)|=|\eta_r(x,y)|=p^{-n_0}$.
The claim is established.\\[2.3mm]
We now show that $r_m=s_m$ for all $m>n_0$ (which completes the proof).
Given~$m$, we claim that, for all $x,y\in A$ with $|x|=1$ and $|y|=p^{-2m}$,
the absolute value
$|\eta_r(x,y)-\eta_r(y,x)|$
equals $p^{-m}$ if $r_m\not=0$,
but is $<p^{-m}$ if $r_m=0$.
If this is true, then the left-hand side of the identity
\[
|\eta_r(b^{-1}t^0,b^{-1}t^{2m})-\eta_r(b^{-1}t^{2m},b^{-1}t^0)|=
|a\eta_s(t^0,t^{2m})| =p^{-m}|s_m|
\]
has absolute value $p^{-m}$ if and only if $r_m\not=0$.
As the right-hand side has absolute value~$p^{-m}$
if and only if $s_m\not=0$, we deduce that $r_m=s_m$.\\[2.3mm]
To establish the claim, write $x=\sum_{i=0}^\infty x_it^i$
and $y=\sum_{j=2m}^\infty y_jt^j$
with $x_i,y_j \in \bF_p$ and $x_0\not=0$ as well as $y_{2m}\not=0$.
For $j\in\N$ such that $j<2m$, set $y_j:=0$.
Then
\begin{equation}\label{ingr8}
\eta_r(x,y)=\eta_r(x_0t^0,y_{2m}t^{2m})+\eta_r(x_0t^0,y')+\eta_r(x',y)
\end{equation}
with $x':=\sum_{i=1}^\infty x_it^i$ and $y':=\sum_{j=2m+1}^\infty y_jt^j$.
The first summand in (\ref{ingr8}) is $t^m r_m x_0y_{2m}$;
its absolute value is $p^{-m}$ if $r_m\not=0$ and vanishes if $r_m=0$. 
Since $\omega_{2n}(x_0t^0,y')=0$ for $n\in\{1,\ldots, m\}$,
the second summand is
\[
\sum_{n=m+1}^\infty r_nt^n\omega_{2n}(x_0t^0,y')
=\sum_{n=m+1}^\infty r_nt^nx_0 y_{2n},
\]
whence it has absolute value $\leq p^{-m-1}<p^{-m}$.
Since $y_{i+2n}\not=0$ implies $i+2n\geq 2m$
and thus $i\geq 2(m-n)$,
the third summand is
\begin{equation}\label{helpsexpla}
\sum_{n=n_0}^\infty\sum_{i=1}^\infty
r_nt^{n+i} x_iy_{i+2n}
=
\sum_{n=n_0}^\infty\sum_{i=\max\{1,2(m-n)\}}^\infty
r_nt^{n+i} x_iy_{i+2n}.
\end{equation}
For $n\in \{n_0,\ldots, m-1\}$, we have $2(m-n)\geq 2$
and thus $\max\{1,2(m-n)\}=2(m-n)$, entailing that
$n+i\geq n+2(m-n)=2m-n\geq 2m-(m-1)=m+1$ and thus
$|t^{n+i}|\leq p^{-m-1}$
on the right hand side of~(\ref{helpsexpla}).
For $n\geq m$, we have $2(m-n)\leq 0$
and thus $\max\{1,2(m-n)\}=1$, whence $n+i\geq n+1\geq m+1$
and hence $|t^{n+i}|\leq p^{-m-1}$.
As a consequence, the third summand in~(\ref{ingr8})
has absolute value $\leq p^{-m-1}$.
Applying~(\ref{ultra}) to the summands in~(\ref{ingr8}),
we see that
\begin{equation}\label{twocase}
\mbox{$|\eta_r(x,y)|=p^{-m}$ if $r_m\not=0$,
while $|\eta_r(x,y)|\leq p^{-m-1}<p^{-m}$ if $r_m=0$,}
\end{equation}
by the ultrametric inequality.
By Lemma~\ref{main-estimate}, we have
\begin{equation}\label{prefin}
|\eta_r(y,x)|<|y|=p^{-2m}<p^{-m}.
\end{equation}
Using (\ref{ultra}) and the ultrametric inequality,
the claim follows from~(\ref{twocase}) and~(\ref{prefin}).\Punkt
\section{Open Problems}\label{secnilp}
If~$p$ is a prime and $(G,\alpha)$ is a
torsion contraction group all of whose composition factors are isomorphic
to $\bF_p(\!(t)\!)$, then $G$ is soluble (as all composition factors are abelian)
and~$G$ has a compact open subgroup which is a pro-$p$-group
and hence pro-nilpotent. The following question is natural:\\[4mm]
\emph{Question} 1. Let $p$ be a prime and $(G,\alpha)$ be a torsion contraction group all of whose composition factors
are isomorphic to $\bF_p(\!(t)\!)$. Will $G$ be nilpotent?\\[2.3mm]
The answer is positive if $G$ is an analytic Lie group over a local field
of positive characteristic and $\alpha$ an analytic automorphism~\cite{Glo},
but the general case is open. The case of two composition factors
should be tackled first:\footnote{The authors are grateful to P.-E. Caprace
for prompting a similar question.}\\[4mm]
\emph{Question} 2. Let $p$ be a prime and $(\wh{G},\alpha)$ be a torsion contraction group
which is an extension $\{e\}\to \bF_p(\!(t)\!)\to \wh{G}\to \bF_p(\!(t)\!)\to\{e\}$.
Will $\wh{G}$ be nilpotent?\\[4mm]
We mention some immediate observations.
\begin{prop}\label{reformu}
The following conditions are equivalent.
\begin{itemize}
\item[{\rm(a)}]
There exists an extension $\{e\}\to \bF_p(\!(t)\!)\to \wh{G}\to\bF_p(\!(t)\!)\to
\{e\}$ of contraction groups such that $\wh{G}$ is not nilpotent.
\item[{\rm(b)}]
There exists an extension $\{e\}\to\bF_p(\!(t)\!)\stackrel{\iota}{\to}
\wh{G}\to\bF_p(\!(t)\!)\to\{e\}$ of contraction groups
such that $\wh{G}$ has
trivial centre, $Z(\wh{G})=\{e\}$.
\item[{\rm(c)}]
There exists a continuous homomorphism
\[
\gamma\colon (\bF_p(\!(t)\!),+)\to
\Aut(\bF_p(\!(t)\!),+)
\]
to the topological group of all automorphisms of the locally compact group $(\bF_p(\!(t)\!),+)$ such that
\begin{equation}\label{equihere}
\gamma(tx)(ty)=t\gamma(x)(y)\;\;
\mbox{for all $x,y\in\bF_p(\!(t)\!)$}
\end{equation}
and $\gamma(z)\not=\id$ for some $z\in\bF_p(\!(t)\!)$.
\end{itemize}
\end{prop}
\begin{proof}
Since every non-trivial nilpotent group has non-trivial centre,
(b)~implies~(a).
By Lemma~\ref{possible-c}, (a) implies~(b).
Next, (b) implies (c) since
$\wh{G}$ as in (b) has $\iota(\bF_p(\!(t)\!))\not\sub Z(\wh{G})$ and hence
induces a non-trivial continuous homomorphism
$\gamma\colon\bF_p(\!(t)\!)\to\Aut(\bF_p(\!(t)\!))$
as described in Example~\ref{moti}
(which satisfies (\ref{equihere})
due to~(\ref{cov2})).

(c)$\impl$(b): Let $\gamma$ be as in~(c).
Then $\bF_p(\!(t)\!)\times \{0\}$ is not contained in the centre of the semi-direct product
$\bF_p(\!(t)\!)\rtimes_\gamma \bF_p(\!(t)\!)$, whence the latter has trivial centre
(by Lemma~\ref{possible-c}).\footnote{Note that the semi-direct product is a contraction group with contractive automorphism $(x,y)\mto (tx,ty)$.}
\end{proof}
\begin{rem}\label{firema}
If $\gamma\colon \bF_p(\!(t)\!)\to\Aut(\bF_p(\!(t)\!))$
is a continuous homomorphism such that~(\ref{equihere}) holds,
define $\theta:=\gamma(t^0)$.
Then
\begin{equation}\label{limitcompo}
\gamma\left(\sum_{n=N}^\infty a_nt^n\right)=\lim_{n\to\infty}\gamma(t^n)^{a_n}\circ\cdots\circ \gamma(t^N)^{a_N}
\end{equation}
in $\Aut(\bF_p(\!(t)\!))$, for all $N\in\Z$ and $a_N,a_{N+1},\ldots \in\bF_p$. Moreover,
\begin{equation}\label{theothers}
\gamma(t^n)(x)=t^n\theta(t^{-n}x)\;\;\mbox{for all $n\in\Z$ and $x\in\bF_p(\!(t)\!)$.}
\end{equation}
Thus $\theta=\id$ implies $\gamma(x)=\id$ for all $x\in\bF_p(\!(t)\!)$.
Since $\gamma$ is a homomorphism and $\bF_p(\!(t)\!)$ is abelian,
we have $\gamma(t^n)\circ \gamma(t^m)=\gamma(t^m)\circ \gamma(t^n)$
for all $n,m\in\Z$. Moreover, $\gamma(x)^p=\gamma(px)=\gamma(0)=\id$ for all
$x\in \bF_p(\!(t)\!)$.
\end{rem}
\begin{rem}
If $\theta\not=\id$ in Remark~\ref{firema},
then $\theta$ cannot be of the most na\"{i}ve
form,
i.e., it cannot come from a permutation $\pi$ of~$\Z$ in the sense
that
\[
\theta(t^k)=t^{\pi(k)}\;\;\mbox{for all $k\in\Z$.}
\]
In fact, since $\theta\not=\id$ and the subgroup generated by $\{t^k\colon k\in\Z\}$
is dense in $\bF_p(\!(t)\!)$, we would have $\ell:=\pi(k)\not=k$ for some $k\in\Z$.
Since $\pi^p=\id$, after replacing $k$ by another element in its $\langle\pi\rangle$-orbit if necessary, we may
assume that $\ell>k$. Set $\Delta:=\ell-k>0$.
Then
\[
\gamma(t^{n\Delta})(t^{n\Delta+k})=\gamma(t^{n\Delta})(t^{n\Delta}t^k)=t^{n\Delta}\theta(t^k)=t^{(n+1)\Delta+k}\;\;\mbox{for all $n\in\N_0$,}
\]
using~(\ref{theothers}).
Thus
\[
(\gamma(t^{n\Delta})\circ\cdots\circ\gamma(t^\Delta)\circ \gamma(t^0))(t^k)=t^{(n+1)\Delta +k}.
\]
For $x:=\sum_{n=0}^\infty t^{n\Delta}\in\bF_p(\!(t)\!)$, we then have
\[
\gamma(x)(t^k)=\lim_{n\to\infty}(\gamma(t^{n\Delta})\circ\cdots\circ\gamma(t^\Delta)\circ \gamma(t^0))(t^k)=\lim_{n\to\infty}t^{(n+1)\Delta +k}=0.
\]
But $\gamma(x)(t^k)\not=0$ as $t^k\not=0$ and $\gamma(x)$ is an
automorphism, contradiction.
\end{rem}
\begin{rem}\label{interestfixed}
Every homomorphism~$\gamma$
as in Proposition~\ref{reformu}\,(c)
satisfies
\[
(\forall 0\not=y\in \bF_p(\!(t)\!))(\exists k\in\Z)\;\; \gamma(t^k)(y)\not=y.
\]
In fact, if there were $0\not=y\in \bF_p(\!(t)\!)$ such that $\gamma(t^k)(y)=y$
for all $k\in\Z$, then we would have $\gamma(x)(y)=y$ for all
$x\in\bF_p(\!(t)\!)$, by (\ref{limitcompo}).
Thus $(y,0)$ would be in the centre of the semidirect
product $\wh{G}:=\bF_p(\!(t)\!)\rtimes_\gamma\bF_p(\!(t)\!)$, since
\[
(a,x)(y,0)(a,x)^{-1}=(\gamma(x)(y),0)=(y,0)\;\;\mbox{for all $(a,x)\in\wh{G}$.}
\]
As a consequence, $\bF_p(\!(t)\!)\times\{0\}$ would be in the centre
of~$\wh{G}$, by Lemma~\ref{possible-c}, entailing that
$\gamma(x)=\id$ for each $x\in \bF_p(\!(t)\!)$.
We have reached a contradiction.
\end{rem}
\begin{rem}
Specializing to $p=2$, consider a continuous homomorphism
$\gamma \colon \bF_2(\!(t)\!)\to\Aut(\bF_2(\!(t)\!))$
such that~(\ref{equihere})
holds.
Then we find a non-zero joint fixed vector~$z$
for $\gamma(t^{k_1}),\ldots, \gamma(t^{k_n})$ in the $\bF_2$-vector space
$\bF_2(\!(t)\!)$
for each finite subset $\{k_1,\ldots, k_n\}\sub \Z$,
as we presently verify.
However, it is not clear whether we can find a non-zero joint fixed vector~$y$
for all $\gamma(t^k)$ with $k\in\Z$ simultaneously.
If $y$ always exists, then $\gamma(x)=\id$ for all $x\in\bF_2(\!(t)\!)$
(by the argument in Remark~\ref{interestfixed}),
whence every extension $\wh{G}$ of $\bF_2(\!(t)\!)$ by itself
would be nilpotent, by Proposition~\ref{reformu}.\\[2.3mm]
For the verification, we use that $\gamma(t^{k_j})^2=\id$ for $j\in\{1,\ldots,n\}$.
Start with a vector $0\not=y_0\in\bF_2(\!(t)\!)$.
We define non-zero elements $y_1,\ldots, y_n\in\bF_2(\!(t)\!)$
such that $y_j$ is a fixed vector for all of $\gamma(t^{k_1}),\ldots,\gamma(t^{k_j})$,
as follows:
Recursively, for $j\in\{1,\ldots, n\}$,
we take $y_j:=y_{j-1}$
if $\gamma(t^{k_j})(y_{j-1})=y_{j-1}$
and
$y_j:=y_{j-1}+\gamma(t^{k_j})(y_{j-1})$ otherwise.
Then $y_j$ is a fixed vector for~$\gamma(t^{k_j})$
and
also for $\gamma(t^{j_1}),\ldots,\gamma(t^{j-1})$,
as $y_{j-1}$ was a fixed vector for these
and $\gamma(t^{k_j})$ leaves the $1$-eigenspace of $\gamma(t^{k_i})$
invariant for all $i\in\{1,\ldots, j-1\}$
since $\gamma(t^{k_i})$ and $\gamma(t^{k_j})$ commute.
Now set $z:=y_n$.
\end{rem}
Nilpotency of contraction groups can be checked on open subgroups.
\begin{prop}\label{nilpvialoc}
Let $\alpha\colon G\to G$
be a contractive automorphism of a
topological
group~$G$.
Then the following holds:
\begin{itemize}
\item[\rm(a)]
If $G$ has a nilpotent open subgroup, then $G$ is nilpotent.
\item[\rm(b)]
If $G$ has a soluble open subgroup, then $G$ is soluble.
\end{itemize}
\end{prop}
Our proof uses a well-known fact (whose simple proof is left to the reader):
\begin{la}\label{ladircomm}
Let $\cA$ and $\cB$ be sets of subgroups of a group~$G$
which are directed upward under inclusion.
Then
$\cC:=\{[A,B]\colon (A,B)\in \cA\times \cB\}$
is directed under inclusion, the unions
$\cup\cA:=\bigcup_{A\in\cA}A$, $\cup\,\cB$, and $\cup\,\cC$
are subgroups of~$G$, and
$[\cup\cA,\cup\cB]=\cup\, \cC$.\Punkt
\end{la}
{\bf Proof of Proposition~\ref{nilpvialoc}.}
Let $U$ be an open subgroup of~$G$ which is nilpotent
(resp., soluble).
Let $n\in\N$ such that $C^n(U)=\{e\}$
(resp., $U^{(n)}=\{e\}$).
Let $\Phi\sub G$ be a finite subset and
$H:=\langle \Phi\rangle$.
Since $\alpha$ is contractive, there is $m\in\N$ such that
$\alpha^m(\Phi)\sub U$ and thus also
$\alpha^m(H)\sub U$.
Then
\[
\alpha^m(C^n(H))\sub C^n(U)\quad\mbox{and}\quad
\alpha^m(H^{(n)})\sub U^{(n)},
\]
showing that $C^n(H)=\{e\}$ (resp.,
$H^{(n)}=\{e\}$).
Since $G$ is the union of its directed set $\cF$ of
finitely generated subgroups, Lemma~\ref{ladircomm}
implies that
\[
C^n(G)=\bigcup_{F\in\cF}C^n(F)=\{e\},\quad\mbox{resp.,}\quad
G^{(n)}=\bigcup_{F\in\cF}F^{(n)}=\{e\}.\vspace{-1mm}
\]
Thus $G$ is nilpotent (resp., soluble).\,\Punkt
\appendix
\section{Extensions of contraction groups with\\
abelian kernel, and cohomology}\label{cohom}
Given a group $G$ and a homomorphism $\gamma\colon G\to\Aut(A)$
to the automorphism group of an abelian group~$A$,
it is well-known that the cohomology classes $[\omega]\in H^2(G,A)$
of $2$-cocycles $\omega\colon G\times G\to A$
parametrize the equivalence classes of those group extensions
\begin{equation}
\{e\}\to A \stackrel{\iota}{\to} \widehat{G}\stackrel{q}{\to} G\to\{e\}
\end{equation}
which \emph{induce $\gamma$} in the sense that
\begin{equation}\label{selection}
\iota(\gamma(q(x))(a))=x\iota(a)x^{-1}\;\;\mbox{for all $a\in A$ and $x\in \widehat{G}$.}
\end{equation}
In this appendix, we give an analogous treatment for extensions
of contraction groups with abelian kernel. We focus on the specific aspects of this
situation and omit calculations which are analogous
to the classical case of group extensions,
as treated in many books (like \cite[Chapter IV.3]{Bro};
cf.\ \cite{Rob} and~\cite{Rot}).
\begin{numba}\label{thesetapp}
Throughout the appendix, $(A,\alpha_A)$ and $(G,\alpha)$
are locally compact contraction groups,
such that~$A$ is abelian.
We fix a continuous homomorphism
\[
\gamma\colon G\to\Aut(A)
\]
to the group of
automorphisms
of the locally compact group~$A$,
such that
\begin{equation}\label{covar}
\gamma(\alpha(g))=\alpha_A\circ \gamma(g)\circ \alpha_A^{-1}\quad\mbox{for all $g\in G$.}
\end{equation}
Then $G\times A\to A$, $(g,a)\mto g.a:=\gamma(g)(a)$ is a continuous action
of~$G$ on~$A$ and (\ref{covar}) is equivalent to the condition
\begin{equation}\label{cov2}
\alpha(g).\alpha_A(a)=\alpha_A(g.a)\quad\mbox{for all $a\in A$ and $g\in G$.}
\end{equation}
The binary group multiplication on~$A$ will be written additively,
while multiplicative notation is used for~$G$.
The neutral elements are $0\in A$ and $e\in G$.
\end{numba}
\begin{defn}
A continuous mapping $\omega\colon G\times G\to A$
is called a \emph{$2$-cocycle of contraction groups}
if it is continuous, satisfies the \emph{cocycle condition}
\begin{equation}\label{cocy}
g_1.\omega(g_2,g_3)-\omega(g_1g_2,g_3)+\omega(g_1,g_2g_3)-\omega(g_1,g_2)=0\;\mbox{for all
$g_1,g_2,g_3\in G$}
\end{equation}
and is \emph{equivariant} in the sense that
\begin{equation}\label{cocyeq}
\omega(\alpha(g_1),\alpha(g_2))=\alpha_A(\omega(g_1,g_2))\quad\mbox{for all $g_1,g_2\in G$.}
\end{equation}
The set $Z^2_{\eq}(G,A)$ of all such~$\omega$
is a subgroup of the abelian group~$A^{G\times G}$.
\end{defn}
\begin{numba}
By the equivariance condition,
$\omega(e,e)=\omega(\alpha^n(e),\alpha^n(e))=\alpha_A^n(\omega(e,e))$ for all $n\in\N$.
As $\alpha_A$ is contractive, letting $n\to\infty$ we find that
\begin{equation}\label{cocnorma}
\omega(e,e)=0,
\end{equation}
i.e., the cocycles we consider are always normalized.
\end{numba}
\begin{numba}
Applying the cocycle condition~(\ref{cocy}) to
$(e,e,g)$, $(g,e,e)$ and $(g^{-1},g,g^{-1})$, respectively, in place of $(g_1,g_2,g_3)$,
we deduce that
\begin{eqnarray}
\omega(e,g)&=& \omega(e,e)\;\;\,=0,\label{helpf1}\\
\omega(g,e)&=&g.\omega(e,e)=0\;\;\mbox{and}\label{helpf2}\\
\omega(g^{-1}, g) &=& g^{-1}\!.\omega(g,g^{-1})\;\;\mbox{for all $g\in G$;}\label{helpf3}
\end{eqnarray}
these identities are useful for the omitted standard calculations.
\end{numba}
\begin{example}\label{moti}
Consider an extension $\{e\}\to A\stackrel{\iota}{\to}\wh{G}\stackrel{q}{\to} G\to\{e\}$
of contraction groups and let $\wh{\alpha}$ be the contractive automorphism of~$\wh{G}$.
Let $\sigma\colon G\to\wh{G}$
be an equivariant continuous section to~$q$, as provided by Theorem~C.
Then
\[
G\times A\to A,\quad g.a:=\iota^{-1}(\sigma(g)\iota(a)\sigma(g)^{-1})
\]
is a continuous left action of~$G$ on~$A$ such that $\gamma(g)\colon A\to A$, $\gamma(g)(a):=g.a$
is an automorphism of~$A$ for each $g\in G$. By \ref{goodStr}, the corresponding homomorphism
\[
\gamma\colon G\to \Aut(A),\quad g\mto \gamma(g)
\]
is continuous, and (\ref{cov2}) holds by construction.
A standard calculation shows that the map
\begin{equation}\label{assococ}
\omega\colon G\times G\to A,\quad (g_1,g_2)\mto
\iota^{-1}(\sigma(g_1)\sigma(g_2)\sigma(g_1g_2)^{-1})
\end{equation}
satisfies the cocycle condition~(\ref{cocy});
as $\omega$ is continuous and inherits~(\ref{cocyeq})
from the equivariance of~$\sigma$, we deduce that $\omega\in Z^2_{\eq}(G,A)$.
The map
\[
\pi\colon A\times G\to \wh{G},\quad (a,g)\mto \iota(a)\sigma(g)
\]
is a bijection. The unique group multiplication on $A\times G$ making
$\pi$ an isomorphism of groups is given by
\begin{equation}\label{theprodu}
(a_1,g_1)(a_2,g_2):= (a_1+g_1.a_2+\omega(g_1,g_2),g_1g_2)
\end{equation}
for $(a_1,g_1),(a_2,g_2)\in A\times G$,
as
\begin{eqnarray*}
\pi(a_1,g_1)\pi(a_2,g_2) &=& \iota(a_1)\sigma(g_1)\iota(a_2)\sigma(g_2)\\
&=&\iota(a_1)\sigma(g_1)\iota(a_2)\sigma(g_1)^{-1}\sigma(g_1)\sigma(g_2)\sigma(g_1g_2)^{-1}\sigma(g_1g_2)\\
&=&\pi(a_1+g_1.a_2+\omega(g_1,g_2),g_1g_2).
\end{eqnarray*}
Since $\pi(0,e)=e$,
the neutral element is $\pi^{-1}(e)=(0,e)$.
As the product topology on $A\times G$ makes~$\pi$ a homeomorphism,
it makes $A\times G$ a topological group, which we denote by $A\times_\omega G$.
The map $\alpha_A\times \alpha$
is a contractive automorphism of $A\times G$, since $\pi\circ (\alpha_A\times\alpha)=
\wh{\alpha}\circ \pi$.
\end{example}
Example~\ref{moti} motivates the following proposition, in
the setting of~\ref{thesetapp}.
\begin{prop}
Given $\omega\in Z^2_{\eq}(G,A)$, the binary operation described
in \emph{(\ref{theprodu})} makes $A\times G$ a
locally compact contraction group $A\times_\omega G$
when endowed with the product topology and the contractive automorphism
$\alpha_A\times \alpha$.
The mappings $\iota\colon A\to A\times_\omega G$, $a\mto (a,e)$
and $\pr_2 \colon A\times_\omega G\to G$, $(a,g)\mto g$
are morphisms of contraction groups and
\begin{equation}\label{exa}
\{e\}\to A\stackrel{\iota}{\to} A\times_\omega G\stackrel{\pr_2}{\to} G\to\{e\}
\end{equation}
is an extension of contraction groups which induces~$\gamma$.
The mapping\linebreak
$\sigma\colon G\to A\times_\omega G$,
$g\mto (0,g)$ is an equivariant continuous section for~$\pr_2$
such that\vspace{-2mm}
\begin{equation}\label{recovom}
\omega(g_1,g_2)=\sigma(g_1)\sigma(g_2)\sigma(g_1g_2)^{-1} \;\;
\mbox{for all $g_1,g_2\in G$}
\end{equation}
and $\iota(\gamma(g)(a))=\sigma(g)\iota(a)\sigma(g)^{-1}$ for all $g\in G$ and $a\in A$.
\end{prop}
\begin{proof}
Classical calculations show that the binary operation (\ref{theprodu})
makes $A\times_\omega G$ a group with neutral element~$(0,e)$,
that $\iota$ and $\pr_2$ are group homomorphisms,
(\ref{recovom}) and the final identity hold,
and that (\ref{exa}) is an exact sequence of groups which induces~$\gamma$.
Inverses are given by
\begin{equation}\label{explinv}
(a,g)^{-1}=(-g^{-1}.a-g^{-1}.\omega(g,g^{-1}),g^{-1})\;\;\mbox{for
$(a,g)\in A\times_\omega G$.}
\end{equation}
Since $A$ and $G$ are topological groups and both $\omega$
and the left $G$-action on~$A$ are continuous,
we deduce from the explicit formulas (\ref{theprodu})
and~(\ref{explinv}) that $A\times_\omega G$ is a topological group.
Using (\ref{cov2}), one finds that the homeomorphism $\alpha_A\times \alpha$
is a homomorphism of groups and hence a contractive automorphism.
To complete the proof, it suffices to observe that $\pr_2\circ (\alpha_A\times \alpha)=\alpha\circ\pr_2$
and $(\alpha_A\times \alpha)\circ\iota=\iota\circ \alpha_A$. 
\end{proof}
\begin{defn}\label{defcobou}
A \emph{$2$-coboundary of contraction groups} is a map of the form
\[
\omega_f\colon G\times G\to A,\quad (g_1,g_2)\mto g_1.f(g_2)-f(g_1g_2)+f(g_1),
\]
where $f\colon G\to A$ is a continuous map which is \emph{equivariant} in the sense that
\begin{equation}\label{fequi}
f\circ\alpha=\alpha_A\circ f.
\end{equation}
\end{defn}
A standard calculation shows that~$\omega_f$ satisfies the cocycle
condition. As it is continuous and inherits the equivariance property~(\ref{cocyeq})
from~(\ref{fequi}), we deduce that~$\omega_f$
is a $2$-cocycle of contraction groups.
Thus, the set $B^2_{\eq}(G,A)$ of all~$\omega_f$ is a subgroup of the abelian
group $Z^2_{\eq}(G,A)$.
\begin{defn} We write
$H^2_{\eq}(G,A):=Z^2_{\eq}(G,A)/B^2_{\eq}(G,A)$
and call $[\omega]:=\omega+B^2_{\eq}(G,A)\in H^2_{\eq}(G,A)$
the \emph{cohomology class} of $\omega\in Z^2_{\eq}(G,A)$.
\end{defn}
We say that two extensions $\{e\}\to A\stackrel{\iota_1}{\to}\wh{G}_1\stackrel{q_1}{\to} G\to\{e\}$
and $\{e\}\to A\stackrel{\iota_2}{\to}\wh{G}_2\stackrel{q_2}{\to} G\to\{e\}$
of contraction groups inducing~$\gamma$ are \emph{equivalent}
if there exists an isomorphism of contraction groups $\psi\colon \wh{G}_1\to\wh{G}_2$
such that $\psi\circ \iota_1=\iota_2$ and $q_2\circ\psi=q_1$.
Given an extension $E\colon \{e\}\to A \stackrel{\iota}{\to} \wh{G}\stackrel{q}{\to} G\to\{e\}$
of contraction groups, let us write $[E]$ for its equivalence class
(or simply $[\wh{G}]$, if $\iota$ and $q$ are understood).
\begin{la}\label{sowelldef}
If $\omega_1,\omega_2\in Z^2_{\eq}(G,A)$ and $\omega_2-\omega_1\in B^2_{\eq}(G,A)$,
then $[A\times_{\omega_1}G]=[A\times_{\omega_2}G]$.
\end{la}
\begin{proof}
We have $\omega_2-\omega_1=\omega_f$ for an equivariant continuous map
$f\colon G\to A$, as in Definition~\ref{defcobou}.
Now $\psi\colon A\times_{\omega_2}G\to A\times_{\omega_1}G$, $(a,g)\mto (a+f(g),g)$
is a homeomorphism and a homomorphism of groups, by standard calculations.
Since~$f$ is equivariant, $\psi$ is a morphism of contraction groups.
\end{proof}
\begin{la}\label{secstocob}
Let $\{e\}\to A\stackrel{\iota}{\to}\wh{G}\stackrel{q}{\to} G\to\{e\}$
be an extension of contraction groups inducing~$\gamma$.
Let $\sigma_j\colon G\to\wh{G}$ be an
equivariant continuous section for~$q$
and $\omega_j(g_1,g_2):=\iota^{-1}(\sigma_j(g_1)\sigma_j(g_2)\sigma_j(g_1g_2)^{-1})$
for $j\in\{1,2\}$ and $g_1,g_2\in G$. Then
$\omega_2-\omega_1\in B^2_{\eq}(G,A)$.
\end{la}
\begin{proof}
The function $f\colon G\to A$,
$f(g):=\iota^{-1}(\sigma_2(g)\sigma_1(g)^{-1})$ is continuous and equivariant.
Standard arguments show that $\omega_2=\omega_1+\omega_f$.
\end{proof}
\begin{prop}\label{soinj}
The assignment $[\omega]\mto[A\times_\omega G]$
is a bijection from $H^2_{\eq}(G,A)$ onto the set of all equivalence
classes of extensions
\[
\{e\}\to A\to\wh{G}\to G\to\{e\}
\]
of contraction groups which induce~$\gamma$.
\end{prop}
\begin{proof}
The assignment is well defined by Lemma~\ref{sowelldef}
and surjective (see Example~\ref{moti}). To prove injectivity,
let $\omega_1,\omega_2\in Z^2_{\eq}(G,A)$ and assume that there exists
an isomorphism of contraction groups
\[
\psi\colon A\times_{\omega_1} G\to A\times_{\omega_2}G
\]
such that $\pr_2\circ\, \psi=\pr_2$ and $\psi\circ \iota=\iota$
(where $\iota\colon A\to A\times G$, $a\mto (a,e)$).
Then both $\sigma_1\colon G\to A\times G$, $g\mto (0,g)$
and $\sigma_2:=\psi^{-1}\circ \sigma_1$ are equivariant sections for~$\pr_2$,
and
\begin{equation}\label{clarify}
\omega_j(g_1,g_2)=\iota^{-1}(\sigma_j(g_1)\sigma_j(g_2)\sigma_j(g_1g_2)^{-1})
\end{equation}
for $j\in\{1,2\}$ and all $g_1,g_2\in G$,
when the products in (\ref{clarify}) are calculated in $A\times_{\omega_1} G$.
Hence $\omega_2-\omega_1\in B^2_{\eq}(G,A)$, by Lemma~\ref{secstocob}.
\end{proof}
{\small{\bf Helge  Gl\"{o}ckner}, Institut f\"{u}r Mathematik, Universit\"at Paderborn,\\
Warburger Str.\ 100, 33098 Paderborn, Germany,
email: {\tt  glockner\at{}math.upb.de};\\[1mm]
\emph{also conjoint professor at} Department of Mathematics,
University of Newcastle, Callaghan, NSW 2308, Australia.\\[3mm]
{\bf George A. Willis}, Department of Mathematics, University of Newcastle,\\
Callaghan, NSW 2308, Australia, email: {\tt George.Willis\at{}newcastle.edu.au}}\vfill
\end{document}